\documentclass[9pt]{amsart}
\usepackage{amsmath,bbm, color}
\usepackage{latexsym,mathrsfs,bm}
\usepackage{mathtools}

\usepackage{url}
\usepackage[english]{babel}
\usepackage{paralist}
\usepackage{amsthm,amsfonts,amssymb,amsmath}
\usepackage[latin1]{inputenc}

\newcommand{\lr}[1]{\left(#1\right)}

\newcommand{\sumstar}{\sideset{}{^*}\sum}

\newcommand{\sumd}{\sideset{}{^d}\sum}

\newcommand{\N}{\mathcal{N}}

\renewcommand{\mod}[1]{~\pr{\textnormal{mod}~#1}}
\newtheorem{thm}{Theorem}[section]

\newtheorem{prop}[thm]{Proposition}

\newtheorem{lemma}[thm]{Lemma}

\theoremstyle{remark}
\newtheorem{rem}{Remark}
\newtheorem{remark}{Remark}
\newtheorem{rem*}{Remark}
\newcommand{\pr}[1]{\left( #1\right)}

\newcommand{\bfrac}[2]{\left(\frac{#1}{#2}\right)}

\def\sumstar{\operatornamewithlimits{\sum\nolimits^*}}

\newcommand{\sumtwo}{\operatorname*{\sum\sum}}
\newcommand{\sumthree}{\operatorname*{\sum\sum\sum}}

\newcommand{\comment}[1]{}

\let\originalleft\left
\let\originalright\right
\renewcommand{\left}{\mathopen{}\mathclose\bgroup\originalleft}
\renewcommand{\right}{\aftergroup\egroup\originalright}

\numberwithin{equation}{section}

\begin{document}

\title{Moment of Derivatives of Quadratic Twists of Modular $L$-functions}

\author{Zijie Zhou}

\begin{abstract}
We prove an asymptotic for the moment of derivatives of quadratic twists of two distinct modular $L$-functions. This was previously known conditionally on GRH by the work of Ian Petrow \cite{Petrow}.
\end{abstract}

\maketitle
\section{Introduction}

Moments of $L$-functions are central objects in the study of analytic number theory. Motivated by the Birch-Swinnerton-Dyer (BSD) conjecture, we want to understand the behavior of the modular $L$-functions at the central value. Analytic methods have been used successfully to study the behavior of the central value in some families of modular $L$-functions; one particular interesting family is that of quadratic twists of modular $L$-functions. According to the BSD conjecture, moments of quadratic twists of modular $L$-functions at the central value can provide crucial information about the distribution of ranks of associated twists of elliptic curves.

Before stating our main result, we introduce some notation and recall some standard facts. Let $f$ be a Hecke new eigenform for the group $\Gamma_0(q)$ of even weight $\kappa$ with trivial central character. It has Fourier expansion
\begin{align*}
f(z) = \sum_{n} \lambda_f(n) n^{(\kappa-1)/2}\exp(2\pi inz).
\end{align*}
The $L$-function associated with $f$ is given by
\begin{align*}
    L(s,f):= \sum_{n}\frac{\lambda_f(n)}{n^s} = \prod_{p\nmid q} \Bigg(1-\frac{\lambda_f(p)}{p^s}+\frac{1}{p^{2s}}\Bigg)^{-1}\prod_{p\mid q} \lr{1-\frac{\lambda_f(p)}{p^s}}^{-1},
\end{align*}
for $\text{Re}(s)>1$. Here, the Hecke eigenvalues of $f$ are all real, and hence $f$ is self-dual. Let $d$ be a fundamental discriminant coprime to $q$ and let $\chi_d(\cdot)=\bfrac{d}{\cdot}$ be a primitive quadratic character of conductor $|d|$. Then $f\otimes\chi_d$ is a newform on $\Gamma(q|d|^2)$ and the twisted $L$-function is defined by
\begin{align*}
L(s,f\otimes \chi_d) := \sum_{n} \frac{\lambda_f(n)\chi_d(n)}{n^s} = \prod_{p\nmid qd}\left(1-\frac{\lambda_f(p)\chi_d(p)}{p^s}+\frac{1}{p^{2s}}\right)^{-1} \prod_{p\mid q} \left(1-\frac{\lambda_f(p)\chi_d(p)}{p^s}\right)^{-1},
\end{align*}
for $\text{Re}(s)>1$. It satisfies the functional equation
\[
\Lambda(s,f\otimes \chi_d) = i^{\kappa}\eta\chi_d(-q)\Lambda(1-s,f\otimes \chi_d)
\]
with
\[
\Lambda(s,f\otimes\chi_d) := \bfrac{|d|\sqrt{q}}{2\pi }^s\Gamma\left(s+\frac{\kappa-1}{2}\right)L(s,f\otimes \chi_d).
\]
Here, $\eta$ is given by the eigenvalues of the Fricke involution, which is independent of $d$ and always $\pm 1$. We also denote the root number by $\omega(f\otimes \chi_d) := i^{\kappa} \eta\chi_d(-q)$.

When $\omega(f\otimes \chi_{d})=1$, an asymptotic formula for the second moment of $L(1/2,f\otimes \chi_d)$ was computed assuming the Generalized Riemann Hypothesis (GRH) by Soundararajan and Young \cite{SY}, which was proved unconditionally in Li's recent work \cite{Li}. Here, we apply their techniques to our problem when the sign of the functional equation is $-1$ and the derivative $L'(1/2,f\otimes \chi_d)$ is the correct object to study instead. For $\text{Re}(s)>1$, the derivative of the twisted $L$-function is
\[
L'(s,f\otimes \chi_d):= - \sum_{n} \frac{\lambda_f(n)\chi_d(n)\log(n)}{n^s}.
\]  
It also has the functional equation
\[
\Lambda'(s,f\otimes \chi_d) = -i^{\kappa}\eta\chi_d(-q)\Lambda'(1-s,f\otimes \chi_d).
\]
Based on Li's work \cite{Li}, Kumar, Mallesham, Sharma and Singh \cite{KMSS} proved the asymptotic of second moment of $L'(1/2,f\otimes\chi_{d})$ unconditionally, which was previously shown by Petrow \cite[Theorem 1]{Petrow} assuming GRH. In this paper, we study the moment of the derivatives of quadratic twists of two distinct modular $L$-functions, which was also shown by Petrow \cite[Theorem 2]{Petrow} assuming GRH.

We let $\sumstar$ denote a sum over squarefree integers 
and $\sumd_{N}$ denote a dyadic sum over $N=2^{\ell}$ for integers $\ell\geq 1$. For convenience, we restrict the modulus to be of the form $8d$ where $d$ is odd and squarefree. Now we state the main result
\begin{thm}\label{thm:One}
Let $f,g$ be distinct normalized cuspidal Hecke newforms with trivial central character, odd levels $q_1,q_2$ and even weights $\kappa_1,\kappa_2$ (we let $Q=q_1q_2$). Let $F$ be a smooth, nonnegative function compactly supported on $[1/2,2]$, we have
\begin{align} \label{eqn:result}
\sumstar_{\substack{(d,2Q)=1\\ \omega(f\otimes \chi_{8d})=-1\\ \omega(g\otimes \chi_{8d})=-1}} L'(1/2,f\otimes\chi_{8d})L'(1/2,g\otimes \chi_{8d})F\left(\frac{8d}{X}\right) = C_0 X (\log X)^2+C'X \log X + O(X(\log\log X)^6).
\end{align}
In the above, we have
\begin{align*}
C_0 & = \frac{\check{F}(0)L(1,\text{sym}^2f)L(1,f\otimes g)L(1,\text{sym}^2g)Z^{*}(0,0)}{2\pi^2},
\end{align*}
and $C'$ is some explicit constant that depends on $f,g$ and $F$. Here, $L(s,f\otimes g)$ is the Rankin-Selberg convolution of two modular forms $f,g$ and $Z^*(u,v)$ is a holomorphic function (\ref{eqn:mainterm}) defined in $\mathrm{Re}(u),\mathrm{Re}(v)\geq -1/4+\epsilon$, depending on $f$ and $g$, given by a sum of four absolutely convergent Euler products and uniformly bounded in $u,v$ where it converges, and $\check{F}$ is a Fourier-type transform of $F$ defined as in (\ref{eqn:fouriertype1}).
\end{thm}

\begin{remark}
The combination of work of Kolyvagin \cite{V}, Gross and Zagier \cite{GZ}, and others, implies that if the analytic rank of an elliptic curve is equal to one, then so is the rank of that elliptic curve. By the modularity theorem, any elliptic curve $E$ over $\mathbb{Q}$ of conductor $N$, there is a weight $2$, level $N$ Hecke eigenform $f$ with $L(s,E) =L(s,f)$, where $L(s,E)$ is the Hasse-Weil $L$-function of the elliptic curve $E$. Our result then implies that for any fixed elliptic curves $E_1$ and $E_2$ over $\mathbb{Q}$ such that $L(s,E_1) = L(s,f)$ and $L(s,E_2) = L(s,g)$, there exist infinitely many fundamental discriminants $d$ such that the twists of the elliptic curves $E_1(d)$ and $E_2(d)$ are exactly of rank $1$.
\end{remark}

\section{Outline of Proof of Theorem \ref{thm:One}}
 Let $M = X/ (\log X)^{100}$, we let
\begin{align*}
\mathcal{A}_f  & \approx  \sum_{n\leq M} \frac{\lambda_f(n)\chi_{8d}(n)}{\sqrt{n}}W_f\bfrac{n}{M}  \\
\mathcal{B}_f &  \approx \sum_{M < n\leq X} \frac{\lambda_f(n)\chi_{8d}(n)}{\sqrt{n}}\Bigg(W_f\bfrac{n}{X} - W_f\bfrac{n}{M}\Bigg),
\end{align*}
where $W_f(x)$ is the smooth cut-off function defined as in (\ref{eqn:cutofffunction}). With the notation above, the approximate functional equation gives us the decomposition
\begin{align*}
\sumstar_{\substack{(d,2Q)=1\\ d\asymp X}} L'(1/2,f\otimes\chi_{8d})L'(1/2,g\otimes\chi_{8d}) = \sumstar_{\substack{(d,2Q)=1\\ d\asymp X}}  \mathcal{A}_f\mathcal{A}_g +\sumstar_{\substack{(d,2Q)=1\\ d\asymp X}} \mathcal{A}_f\mathcal{B}_g+ \sumstar_{\substack{(d,2Q)=1\\ d\asymp X}}  \mathcal{B}_f\mathcal{A}_g+\sumstar_{\substack{(d,2Q)=1\\ d\asymp X}} \mathcal{B}_f\mathcal{B}_g.
\end{align*}

When one of the $n_i$ sums is small, we can apply Poisson summation to obtain both diagonal and off-diagonal terms. The diagonal term gives us the main contributions while we need to bound the off-diagonal terms with Lemma \ref{lem:NonDyadicCut} from \cite{Li}. The main challenge is to bound the tail
\begin{align*}
\sumstar_{\substack{(d,2Q)=1\\ d\asymp X}}\mathcal{B}_f\mathcal{B}_g
\end{align*}
since that is when both $n_i$ sums are large. Our idea behind bounding the tail is to insert a smooth partition of unity for each $n_i$ sum which leads us to analyze the following
\begin{align} \label{eqn:Dyadic}
\sumd_{N_1,N_2} \sumstar_{\substack{(d,2Q)=1\\ d\asymp X}} \sumtwo_{n_1,n_2} \frac{\lambda_f(n_1)\lambda_g(n_2)}{\sqrt{n_1n_2}}\chi_{8d}(n_1n_2)\Bigg(W_f\bfrac{n_1}{X} - W_f\bfrac{n_1}{M}\Bigg)\Bigg(W_g\bfrac{n_2}{X} - W_g\bfrac{n_2}{M}\Bigg)G\bfrac{n_1}{N_1}G\bfrac{n_2}{N_2},
\end{align}
where $G(x)$ is a smooth real-valued function that satisfies (\ref{eqn:G}). Next we truncate each dyadic sum $N_i$ in three different ranges: $N_i\leq M, M<N_i\leq X$ and $X<N_i$. 

In the work of \cite{KMSS}, the authors apply Cauchy-Schwarz inequality for all ranges which caused them to lose an extra power of $(\log X)^2$ in the bound for $\sumstar_{d\asymp X} \mathcal{B}_f^2$. But such loss is acceptable given the main term is $X(\log X)^3$ in their work. This is illustrated in the table below.
        
\begin{align*}
            \begin{array}{|c|c|c|c|}
            \hline
                & N_1\leq M & M<N_1\leq X & X<N_1 \\
                \hline
          N_2\leq M & \text{C.S.} & \text{C.S.} & \text{C.S.} \\
                \hline
    M <N_2\leq X  & \text{C.S.} & \text{C.S.} & \text{C.S.} \\
                \hline
            X<N_2  &   \text{C.S.} & \text{C.S.} & \text{C.S.}  \\
            \hline
            \end{array}
  \end{align*}.




On the other hand, our approach takes advantage when one of the dyadic sums is small by using Poisson summation instead of Cauchy-Schwarz ineqaulity. For the ranges $M<N_i\leq X$, we observe that the length is small in the $\log$ scale. That is, there are $\log\log X$ dyadic intervals in that range so we can apply Cauchy-Schwarz inequality and Lemma \ref{lem:NonDyadicCut} to obtain the bound. For the range $X<N_i$, the smooth decay from the cut-off function will also give us small contribution. As a result, we are able to save a power of $(\log X)^2$ at the end. Our approach is illustrated in the table below.
\begin{align*}
            \begin{array}{|c|c|c|c|}
            \hline
                & N_1\leq M & M <N_1\leq X & X<N_1 \\
                \hline
          N_2\leq M & \textcolor{red}{P.S.} & \textcolor{red}{P.S.} & \textcolor{red}{P.S.} \\
                \hline
    M <N_2\leq X  & \textcolor{red}{P.S.} & \text{C.S.} & \text{C.S.} \\
                \hline
            X<N_2  &   \textcolor{red}{P.S.} & \text{C.S.} & \text{C.S.}  \\
            \hline
            \end{array}
  \end{align*}.
  \begin{remark}
  Alternatively, one can also redefine $\mathcal{A}_f,\mathcal{B}_f$ in the following ways
  \begin{align*}
  \mathcal{A}_f & := \sumd_{N\leq M} \sum_{n} \frac{\lambda_f(n)\chi_{8d}{(n)}}{\sqrt{n}} G\bfrac{n}{N}W_f\bfrac{n}{M}, \\
  \mathcal{B}_f & := L'(1/2,f\otimes\chi_{8d}) - \mathcal{A}_f.
  \end{align*}
  This simplifies the bound for $\sumstar_{d} \mathcal{B}_f\mathcal{B}_g$ but it takes more work to extract the main terms.
  \end{remark}

{\bf Acknowledgments.} I would like to thank my advisor Xiannan Li for the fruitful conversations and helpful comments. I also would like to thank the anonymous referee for the detailed feedback.

\section{Preliminary Results}
By the approximate functional equation, $L'(1/2,f\otimes\chi_{d})$ can be expressed in terms of a Dirichlet polynomial. We record a standard result from Petrow \cite[Lemma~1]{Petrow}:
\begin{lemma}(Approximate functional equation). \label{lem:ApproximateFunctionalEquation}
Let $f$ be a $\lambda_f(1)=1$ normalized cuspidal newform on $\Gamma_0(q)$ with trivial central character and root number $\omega(f)= i^{\kappa}\eta$. Let $Z>0$ be an arbitrary real number parameter. Define the cut-off function
\begin{equation*}
W_Z(x): = \frac{1}{2\pi i}\int_{(3)}\frac{\Gamma(u+\kappa/2)}{\Gamma(\kappa/2)}\left(\frac{2\pi x}{Z\sqrt{q}}\right)^{-u}\frac{1-u\log Z}{u^2}\; du.
\end{equation*}
Then
\[
\sum_{n}\frac{\lambda_f(n)\chi_d(n)}{\sqrt{n}}W_Z\left(\frac{n}{|d|}\right)-i^{\kappa}\eta\chi_d(-q) \sum_{n}\frac{\lambda_f(n)\chi_d(n)}{\sqrt{n}}W_{Z^{-1}}\left(\frac{n}{|d|}\right) =
\begin{cases}
L'(1/2,f\otimes\chi_{d}), & \text{if} \;\omega(f\otimes\chi_d)=-1, \\
0, & \text{if} \;\omega(f\otimes\chi_d)=1.
\end{cases}
\]

\end{lemma}

For simplicity, we let $Z=1$ and $\gamma_f(u) := \frac{\Gamma(u+\kappa/2)}{\Gamma(\kappa/2)} \bfrac{2\pi}{\sqrt{q}}^{-u}$ throughout the paper. 
An application of M\"obius inversion will lead us to analyze a character sum in the form of
\begin{align*}
    \sum_{(d,2)=1} \chi_{8d}(n_1n_2)F\bfrac{8d}{X}.
\end{align*}
The following lemma is a refinement of Poisson summation from Soundararajan \cite[Lemma 2.6]{Snonvanishing}.
\begin{lemma}(Poisson Summation). \label{lem:PoissonSummation}
Let $F$ be a smooth function compactly support on the positive real numbers, and suppose that $n$ is an odd integer. Then
\[
\sum_{(d,2)=1} \bfrac{d}{n}F\bfrac{d}{Z} = \frac{Z}{2n}\bfrac{2}{n}\sum_{k\in\mathbb{Z}} (-1)^k G_k(n)\check{F}\bfrac{kZ}{2n},
\]
where 
\[
G_k(n) = \left(\frac{1-i}{2}+\bfrac{-1}{n}\frac{1+i}{2}\right)\sum_{a\mod n}\bfrac{a}{n}e\bfrac{ak}{n}
\]
is the Gauss-type sum, and
\begin{align} \label{eqn:fouriertype1}
\check{F}(y) = \int_{-\infty}^{\infty} (\cos(2\pi xy)+\sin(2\pi xy))F(x)\; dx
\end{align}
is a Fourier-type transform of $F$. While for $F$ supported on $[0,\infty)$, we also have
\begin{equation} \label{eqn:fouriertype}
\check{F}(y) =\frac{1}{2\pi i}\int_{(1/2)} \tilde{F}(1-s)\Gamma(s)(\cos+\mathrm{sgn}(y)\sin)\bfrac{\pi s}{2}(2\pi |y|)^{-s}\; ds,
\end{equation}
where 
\[
\tilde{F}(s)=\int_{0}^{\infty} F(x)x^{s-1}\; dx
\]
is the usual Mellin transform of $F$.
\end{lemma}
Furthermore, Soundararajan has computed $G_k(n)$ explicitly in \cite[Lemma 2.3]{Snonvanishing}.
\begin{lemma}\label{lem:GaussSum}
If $m$ and $n$ are relatively prime odd integers, then $G_k(mn) = G_k(m)G_k(n)$, and if $p^{\alpha}$ is the largest power of $p$ dividing $k$ (setting $\alpha=\infty$ if $k=0$), then
\[
G_k(p^{\beta}) =
\begin{cases}
0, & \text{if $\beta\leq \alpha$ is odd}, \\
\phi(p^{\beta}), & \text{if $\beta\leq \alpha$ is even}, \\
-p^{\alpha}, & \text{if $\beta = \alpha+1$ is even}, \\
\bfrac{kp^{-\alpha}}{p}p^{\alpha}\sqrt{p}, & \text{if $\beta= \alpha+1$ is odd}, \\
0, & \text{if $\beta\geq \alpha+2$}. \\
\end{cases}
\]
\end{lemma}

After applying Poisson summation, we obtain both diagonal and off-diagonal terms. While the diagonal term gives us the desired asymptotic, we will encounter the Dirichlet series in the form of 
\begin{align*}
Z(\alpha,\beta,\gamma;a,k_1,Q',\chi_b) := \sum_{\substack{k_2\geq 1\\ (k_2,2)=1}} \sumtwo_{\substack{(n_1n_2,2a)=1}} \frac{\lambda_f(n_1)\lambda_g(n_2)\chi_b(n_1n_2Q')}{n_1^{\alpha}n_2^{\beta}k_2^{2\gamma}} \frac{G_{k_1k_2^2}(n_1n_2Q')}{n_1n_2Q'}
\end{align*}
when bounding the off-diagonal terms. The series is a variation of $(3.10)$ in \cite{SY} and the following lemma is a refinement of \cite[Lemma 3.3]{SY}, we will provide a proof of it for the sake of completeness.

\begin{lemma}\label{lem:DirichletSeries}
With the notation above, let $f_{b} := f\otimes \chi_{b},g_{b}: =g\otimes \chi_{b}$ be the twisted newforms. Let $k_1$ be squarefree and
\[
m(k_1) =
\begin{cases}
k_1, & \quad\text{if $k_1\equiv 1\pmod{4}$}, \\
4k_1, &\quad \text{if $k_1\equiv 2,3\pmod{4}$}. \\
\end{cases}
\]
Then for any positive integers $a, Q'$, we have
\[
Z(\alpha,\beta,\gamma; a,k_1,Q',\chi_b) = L(1/2+\alpha,f_{b}\otimes \chi_{m(k_1)}) L(1/2+\beta,g_{b}\otimes \chi_{m(k_1)})Y(\alpha,\beta,\gamma;a,k_1,Q',\chi_b)
\]
with
\[
Y(\alpha,\beta, \gamma;a,k_1,Q',\chi_b) = \frac{Z^{*}(\alpha,\beta,\gamma;a,k_1,Q',\chi_b)}{L(1+2\alpha,\text{sym}^2f) L(1+2\beta,\text{sym}^2g) L(1+\alpha+\beta,f\otimes g)},
\]
 where $Z^{*}(\alpha,\beta,\gamma;a,k_1,Q',\chi_b)$ is analytic in the region $\mathrm{Re}(\alpha),\mathrm{Re}(\beta)\geq -\delta/2$ and $\mathrm{Re}(\gamma)\geq 1/2+\delta$ for any $0<\delta <1/3$. Moreover, in the same region, $Z^{*}(\alpha,\beta,\gamma;a,k_1,Q',\chi_b)\ll \tau(a)$ where the implied constant may depend only on $\delta, f$ and $g$. Furthermore, we can represent
 \begin{align*} 
 Y(\alpha,\beta,\gamma;a,k_1,Q',\chi_b) = \sumthree_{r_1,r_2,r_3} \frac{C(r_1,r_2,r_3)}{r_1^{\alpha}r_2^{\beta}r_3^{2\gamma}}
 \end{align*}
as a Dirichlet series with some coefficients $C(r_1,r_2,r_3)$. 
\end{lemma}

\begin{proof}
The terms of the Dirichlet series defining $Z$ are jointly multiplicative in $k_2,n_1$ and $n_2$, so that we can decompose $Z$ as a Euler product. Suppose $p^{r_p}$ is the largest power of $p$ dividing $Q'$, then the generic Euler factor is given by
\begin{align*}
\sumthree_{k_2,n_1,n_2\geq 0} \frac{\lambda_f(p^{n_1})\lambda_g(p^{n_2})\chi_b(p^{n_1+n_2+r_p})}{p^{n_1\alpha+n_2\beta+2k_2\gamma}} \frac{G_{k_1p^{2k_2}}(p^{n_1+n_2+r_p})}{p^{n_1+n_2+r_p}}.
\end{align*}

For $p=2$, the Euler factor is $1$.

For $p\mid a$ and $p\mid  k_1$, $n_1,n_2$-sums do not contribute so the Euler factor is 
\[
\sum_{k_2\geq 0} \frac{\chi_b(p^{r_p})}{p^{2k_2\gamma}} \frac{G_{k_1p^{2k_2}}(p^{r_p})}{p^{r_p}} =
\begin{cases}
\left(1-\frac{1}{p^{2\gamma}}\right)^{-1}, & \text{if $r_p= 0$}, \\
0,  & \text{if $r_p = 1$}, \\
-\frac{1}{p}+\frac{1}{p^{2\gamma}} + O\bfrac{1}{p^{1+2\gamma}}, & \text{if $r_p=2$}, \\
O\bfrac{1}{p^{1+2\delta}}, & \text{if $r_p>2$}.
\end{cases}
\]

For $p\mid a$ but $p \nmid k_1$, the Euler factor is
\[
\sum_{k_2\geq 0} \frac{\chi_b(p^{r_p})}{p^{2k_2\gamma}} \frac{G_{k_1p^{2k_2}}(p^{r_p})}{p^{r_p}} =
\begin{cases}
\left(1-\frac{1}{p^{2\gamma}}\right)^{-1}, & \text{if $r_p= 0$}, \\
\frac{\chi_{bk_1}(p)}{\sqrt{p}},  & \text{if $r_p = 1$}, \\
O\bfrac{1}{p^{1+2\delta}}, & \text{if $r_p>1$}.
\end{cases}
\]

For $p\nmid a$ but $p\mid k_1$, the Euler factor is
\[
\sumthree_{k_2,n_1,n_2\geq 0} \frac{\lambda_f(p^{n_1}) \lambda_g(p^{n_2}) \chi_b(p^{n_1+n_2+r_p}) }{p^{n_1\alpha+n_2\beta+2k_2\gamma}} \frac{G_{k_1p^{2k_2}}(p^{n_1+n_2+r_p})}{p^{n_1+n_2+ r_p}} =
\begin{cases}
1-\frac{1}{p}\left(\frac{\lambda_f(p^2)}{p^{2\alpha}} +\frac{\lambda_f(p)\lambda_g(p)}{p^{\alpha+\beta}} +\frac{\lambda_g(p^2)}{p^{2\beta}}\right) + O\bfrac{1}{p^{1+2\delta}}, & \text{if $r_p= 0$}, \\
-\frac{1}{p}\left(\frac{\lambda_f(p)}{p^{\alpha}}+\frac{\lambda_g(p)}{p^{\beta}}\right) + O\left(\frac{1}{p^{1+2\delta}}\right),  & \text{if $r_p = 1$}, \\
O\left(\frac{1}{p^{1+2\delta}}\right), & \text{if $r_p > 1$}.
\end{cases}
\]

Lastly, for $p\nmid ak_1$, the Euler factor is
\[
\sumthree_{k_2,n_1,n_2\geq 0} \frac{\lambda_f(p^{n_1}) \lambda_g(p^{n_2}) \chi_b(p^{n_1+n_2+r_p}) }{p^{n_1\alpha+n_2\beta+2k_2\gamma}} \frac{G_{k_1p^{2k_2}}(p^{n_1+n_2+r_p})}{p^{n_1+n_2+ r_p}} =
\begin{cases}
1+\frac{\chi_{b}(p)}{p^{1/2}}\left(\frac{\lambda_f(p)\chi_{k_1}(p)}{p^{\alpha}}+\frac{\lambda_g(p)\chi_{k_1}(p)}{p^{\beta}}\right) + O\bfrac{1}{p^{1+2\delta}}, & \text{if $r_p= 0$}, \\
\frac{\chi_b(p)\chi_{k_1}(p)}{p^{1/2}} + O\left(\frac{1}{p^{1+2\delta}}\right), & \text{if $r_p = 1$}, \\
O\left(\frac{1}{p^{1+2\delta}}\right), & \text{if $r_p > 1$}.
\end{cases}
\]
Combine all of the Euler factors together, we have the desired result.
\end{proof}

Lastly, we need to introduce a smooth real-valued function $G$ which is compactly supported on $[3/4,2]$ and it satisfies
\begin{align*}
G(x) & = 1 \;\text{for all $x\in [1,3/2]$}, \\
G(x)+G(x/2)&  = 1\; \text{for all $x\in [1,3]$}.
\end{align*} 
Functions like $G$ appear in standard constructions of partitions of unity and we refer the reader to Warner's book \cite[Lemma~1.10]{Warner} for more details. It can be verified that
\begin{align} \label{eqn:G}
G(x)+G(x/2)+\cdots +G(x/2^{J}) = 1
\end{align}
for $x\in [1, 3\cdot 2^{J-1}]$ and is supported on $[3/4, 2^{J+1}]$. We fix, once and for all, a function $G$ with the properties above. We now record an important lemma from Li \cite[Lemma 6.3]{Li}.

\begin{lemma} \label{lem:NonDyadicCut}
For any real $\mathcal{X}, N\geq 1$, real $t$, and positive integer $q$
\[
\sum_{\substack{(d,2)=1 \\ d\leq \mathcal{X}}} \left|\sum_{(n,q)=1} \frac{\lambda_f(n)}{n^{1/2+it}}\bfrac{8d}{n}G\bfrac{n}{N}\right|^2 \ll \tau(q)^5\mathcal{X} (1+|t|)^3\log(2+|t|),
\]
where $\tau(n)$ is the divisor function. 
\end{lemma}
When bounding the off-diagonal terms later, we need to use dyadic version of the lemma above
\begin{lemma} \label{lem:NonDyadicCut2}
Let $G(x)$ be the smooth function defined above, we have
\begin{align*} 
\sumstar_{M\leq |m|\leq 2M} \left|\sum_{n=1}^{\infty} \frac{\lambda_f(n)}{n^{1/2+it}}\bfrac{m}{n}G\bfrac{n}{N}\right|^2 \ll_{f} (1+|t|)^2 \Big( M+N\log\Big(2+\frac{N}{M}\Big)\Big).
\end{align*}
\end{lemma}

\section{Main Propositions}
Recall $M = X/(\log X)^{100}$, we define
\begin{align} \label{eqn:head}
\mathcal{A}_f: &= \mathcal{A}(1/2,f\otimes\chi_{8d})=(1-i^{\kappa_1}\eta_f\chi_{8d}(-q_1)) \sum_{n} \frac{\lambda_f(n)\chi_{8d}(n)}{\sqrt{n}}W_f\bfrac{n}{M}  
\end{align}
with 
\begin{align} \label{eqn:cutofffunction}
W_f(x):=\frac{1}{2\pi i}\int_{(3)}\frac{\Gamma(u+\kappa_1/2)}{\Gamma(\kappa_1/2)}\left(\frac{2\pi x}{\sqrt{q_1}}\right)^{-u} \frac{1}{u^2}\; du = \frac{1}{2\pi i}\int_{(3)} \gamma_f(u) x^{-u} \frac{1}{u^2} \; du.
\end{align}
We also define $\mathcal{B}_f: = \mathcal{B}(1/2,f\otimes\chi_{8d})$ by setting $\mathcal{B}(1/2,f\otimes\chi_{8d})= L'(1/2,f\otimes \chi_{8d})  - \mathcal{A} (1/2,f\otimes\chi_{8d})$.
With the notation above, we rewrite
\begin{align*}
\sumstar_{\substack{(d,2Q)=1 \\ \omega(f\otimes\chi_{8d})=-1 \\ \omega(g\otimes\chi_{8d}) = -1 }} L'(1/2,f\otimes \chi_{8d})L'(1/2,g\otimes\chi_{8d})F\left(\frac{8d}{X}\right) = \text{I}(f,g) + \text{II}(f,g) + \text{II}(g,f) + \text{III}(f,g),
\end{align*}
where
\begin{align*}
\text{I}(f,g) & := \sumstar_{\substack{(d,2Q)=1 \\ \omega(f\otimes\chi_{8d})=-1 \\ \omega(g\otimes\chi_{8d}) = -1 }}  \mathcal{A}_f\mathcal{A}_gF\left(\frac{8d}{X}\right), \\
\text{II}(f,g) & :=  \sumstar_{\substack{(d,2Q)=1 \\ \omega(f\otimes\chi_{8d})=-1 \\ \omega(g\otimes\chi_{8d}) = -1 }}  \mathcal{A}_f\mathcal{B}_gF\left(\frac{8d}{X}\right), \\
\text{II}(g,f) & :=  \sumstar_{\substack{(d,2Q)=1 \\ \omega(f\otimes\chi_{8d})=-1 \\ \omega(g\otimes\chi_{8d}) = -1 }} \mathcal{A}_g\mathcal{B}_f F\left(\frac{8d}{X}\right), \\
\text{III}(f,g) & := \sumstar_{\substack{(d,2Q)=1 \\ \omega(f\otimes\chi_{8d})=-1 \\ \omega(g\otimes\chi_{8d}) = -1 }}  \mathcal{B}_f\mathcal{B}_gF\left(\frac{8d}{X}\right).
\end{align*}
Theorem \ref{thm:One} immediately follows from the following propositions:
\begin{prop}\label{prop:First}
\begin{align*}
\mathrm{I}(f,g) = C_0 X (\log M)^2+C_1X\log M+O(X)
\end{align*}
with $C_1$ is defined as in (\ref{eqn:C1}).
\end{prop}
\begin{prop}\label{prop:Second}
\begin{align*}
\mathrm{II}(f,g) = C_0 X\log M \log\bfrac{X}{M}+ CX\log M+ O(X\log\log X)
\end{align*}
for some explicit constant $C$ that depends on $f,g$ and $F$.
\end{prop}
\begin{prop}\label{prop:Third}
\begin{align*}
\mathrm{III}(f,g) \ll X (\log\log X)^6.
\end{align*}
\end{prop}

\section{Proof of Proposition \ref{prop:First}}
Let $Q'=1,q_1,q_2$ or $q_1q_2$, we set $h(n,m,d) = W_f\bfrac{n}{M}W_g\bfrac{m}{M}F\bfrac{8d}{X}$ and define
\begin{align} \label{eqn:decomposition}
S_{f,g}(Q';h):= \sumstar_{(d,2Q)=1} \sumtwo_{n_1,n_2} \frac{\lambda_f(n_1)\lambda_g(n_2)}{\sqrt{n_1n_2}} \chi_{8d}(n_1n_2Q')h(n_1,n_2,d).
\end{align}
By (\ref{eqn:head}), we write
\begin{align*}
\text{I}(f,g) = S_{f,g}(1;h)-i^{\kappa_1}\eta_fS_{f,g}(q_1;h)-i^{\kappa_2}\eta_gS_{f,g}(q_2;h)+i^{\kappa_1+\kappa_2}\eta_f\eta_gS_{f,g}(Q;h).
\end{align*}
Apply M\"obius inversion to sieve out both squarefree condition and coprime condition with $Q$, we have
\begin{align*}
S_{f,g}(Q';h)& =\sum_{(a,2Q)=1} \mu(a)\sum_{b\mid Q}\mu(b)  \sumtwo_{\substack{(n_1n_2,2a)=1}} \frac{\lambda_f(n_1)\lambda_g(n_2)}{\sqrt{n_1n_2}} W_f\left(\frac{n_1}{M}\right)W_g\left(\frac{n_2}{M}\right)\sum_{(d,2)=1} \chi_{8bd}(n_1n_2Q')F\left(\frac{8a^2bd}{X}\right).
\end{align*} 
Let $Y$ be a real parameter which will be chosen later, we split $S_{f,g}(Q';h)=T_{(a\leq Y)}+ T_{(a>Y)}$ with
\begin{align*}
T_{(a\leq Y)} & =  \sum_{\substack{(a,2Q)=1\\ a\leq Y}} \mu(a) \sum_{b\mid Q}\mu(b) \sumtwo_{\substack{(n_1n_2,2a)=1}} \frac{\lambda_f(n_1)\lambda_g(n_2)}{\sqrt{n_1n_2}} W_f\left(\frac{n_1}{M}\right)W_g\left(\frac{n_2}{M}\right)\sum_{(d,2)=1} \chi_{8bd}(n_1n_2Q')F\left(\frac{8a^2bd}{X}\right),
\end{align*}
and
\begin{align*}
T_{(a>Y)} & =  \sum_{\substack{(a,2Q)=1\\ a> Y}} \mu(a) \sum_{b\mid Q} \mu(b) \sumtwo_{\substack{(n_1n_2,2a)=1}} \frac{\lambda_f(n_1)\lambda_g(n_2)}{\sqrt{n_1n_2}} W_f\left(\frac{n_1}{M}\right)W_g\left(\frac{n_2}{M}\right)\sum_{(d,2)=1} \chi_{8bd}(n_1n_2Q')F\left(\frac{8a^2bd}{X}\right).
\end{align*}

\subsection{Bound of $T_{(a>Y)}$}
Apply Cauchy-Schwarz inequality and using the bound from \cite[Proposition 3]{KMSS}, we have
\begin{align*}
T_{(a>Y)} & \ll \sum_{\substack{(a,2Q)=1\\ a>Y}} \frac{\tau(2a)^5 X(\log X)^3}{a^2}\ll \frac{X(\log X)^3}{Y^{1-\epsilon}}.
\end{align*}

\subsection{Asymptotic of $T_{(a\leq Y)}$} By Lemma \ref{lem:PoissonSummation}, we derive
\begin{align*}
& T_{(a\leq Y)} \\ 
 & = \frac{X}{16} \sum_{\substack{(a,2Q)=1\\ a\leq Y}} \frac{\mu(a)}{a^2} \sum_{b\mid Q}\frac{\mu(b)}{b} \sumtwo_{\substack{(n_1n_2,2a)=1}} \frac{\lambda_f(n_1)\lambda_g(n_2)\chi_b(n_1n_2Q')}{\sqrt{n_1n_2}} W_f\left(\frac{n_1}{M}\right)W_g\left(\frac{n_2}{M}\right) \sum_{k\in\mathbb{Z}} (-1)^k\frac{G_k(n_1n_2Q')}{n_1n_2Q'}\check{F}\left(\frac{kX}{16a^2bn_1n_2Q'}\right).
\end{align*}

\subsubsection{(Diagonal term $k=0$)} From Lemma \ref{lem:GaussSum}, we observe that $G_0(n_1n_2Q') = \phi(n_1n_2Q')$ if $n_1n_2Q' =\square$ and $G_0(n_1n_2Q')= 0$ otherwise. Furthermore, we have
\begin{align} \label{eqn:Mobius}
\sum_{\substack{(a,2n_1n_2Q)=1 \\ a\leq Y}} \frac{\mu(a)}{a^2}  = \frac{8}{\pi^2} \prod_{p\mid n_1n_2Q} \left(1-\frac{1}{p^2}\right)^{-1}+ O(Y^{-1}).
\end{align}
This implies
\begin{align*}
T_{\text{diagonal}} & :=\frac{\check{F}(0)X}{16} \sum_{\substack{(a,2Q)=1\\ a\leq Y}} \frac{\mu(a)}{a^2} \sum_{b\mid Q} \frac{\mu(b)}{b} \sumtwo_{\substack{(n_1n_2,2a)=1\\ n_1n_2Q'=\square}} \frac{\lambda_f(n_1)\lambda_g(n_2)}{\sqrt{n_1n_2}} W_f\left(\frac{n_1}{M}\right)W_g\left(\frac{n_2}{M}\right) \frac{\phi(n_1n_2Q')}{n_1n_2Q'} = \mathcal{M}_{Q'}+\mathcal{E}_{Q'},
\end{align*}
with
\begin{align*}
\mathcal{M}_{Q'} 
& = \frac{\check{F}(0)X}{2\pi^2} \sumtwo_{\substack{(n_1n_2,2a)=1\\ n_1n_2Q'=\square}} \frac{\lambda_f(n_1)\lambda_g(n_2)}{\sqrt{n_1n_2}} W_f\left(\frac{n_1}{M}\right)W_g\left(\frac{n_2}{M}\right) \prod_{p\mid n_1n_2Q} \frac{p}{p+1}
\end{align*}
and
\begin{align*}
\mathcal{E}_{Q'} 
\ll  \frac{X}{Y} \sumtwo_{\substack{(n_1n_2,2a)=1 \\ n_1n_2Q' = \square}} \frac{d(n_1)d(n_2)}{\sqrt{n_1n_2}}\bigg|W_f\bfrac{n_1}{M}W_g\bfrac{n_2}{M}\Bigg| \ll \frac{X(\log X)^{11}}{Y}.
\end{align*}
To analyze the main term $\mathcal{M}_{Q'}$, we use the computation from \cite[p.11]{Petrow}. In particular, we have
\begin{align} \label{eqn:mainterm}
\sumtwo_{\substack{(n_1n_2,2a)=1\\ n_1n_2Q'=\square}} \frac{\lambda_f(n_1)\lambda_g(n_2)}{n_1^{1/2+u}n_2^{1/2+v}} \prod_{p\mid n_1n_2Q} \frac{p}{p+1} = 
L(1+u+v,f\otimes g)L(1+2u,\text{sym}^2f)L(1+2v,\text{sym}^2g)Z_{Q'}^*(u,v),
\end{align}
where $Z_{Q'}^*(u,v)$ is given by some absolutely convergent Euler products which is uniformly bounded in the region $\text{Re}(u),\text{Re}(v)\geq -1/4+\epsilon$. By (\ref{eqn:cutofffunction}) and (\ref{eqn:mainterm}), we deduce
\begin{align*}
\mathcal{M}_{Q'} = \frac{\check{F}(0)X}{2\pi^2}\bfrac{1}{2\pi i}^2 \int_{(3)}\int_{(3)} \gamma_f(u)\gamma_g(v)  \frac{M^u}{u^2}\frac{M^v}{v^2} L(1+u+v,f\otimes g)L(1+2u,\text{sym}^2f)L(1+2v,\text{sym}^2g)Z_{Q'}^{*}(u,v) \; dudv.
\end{align*}
Shift the contours $\text{Re}(u),\text{Re}(v)= 3$ to $\text{Re}(u),\text{Re}(v)=-1/5$, we encounter poles of order $2$ at $u,v=0$. Then Cauchy's Theorem implies 
\begin{align*}
& \mathcal{M}_{Q'}  = C_{Q'}X\left(\log M+\gamma_f'(0) + 2\frac{L'}{L}(1,\text{sym}^2f) \right) \left(\log M+\gamma_g'(0) + \frac{L'}{L}(1,f\otimes g) + 2\frac{L'}{L}(1,\text{sym}^2g) +\frac{\frac{d}{dv}Z_{Q'}^*(0,v)\vert_{v=0}}{Z_{Q'}^*(0,0)}  \right) \\
&  + C_{Q'}\frac{L'}{L}(1,f\otimes g) X \left(\log M+\gamma_g'(0)+\frac{L''}{L'}(1,f\otimes g)+2\frac{L'}{L}(1,\text{sym}^2g)+\frac{\frac{d}{dv}Z_{Q'}^*(u,v)\vert_{v=0}}{Z_{Q'}^*(0,0)}\right) \\ 
& + C_{Q'} \frac{\frac{d}{du}Z_{Q'}^*(u,0)\vert_{u=0}}{Z_{Q'}^*(0,0)} X \left(\log M+\gamma_g'(0)+\frac{L'}{L}(1,f\otimes g)+2\frac{L'}{L}(1,\text{sym}^2g)+\frac{\frac{d}{dv}\frac{d}{du}Z_{Q'}^*(u,v)\vert_{u,v=0}}{\frac{d}{du}Z_{Q'}^*(u,0)\vert_{u=0}}\right)+O(X^{4/5+\epsilon}),
\end{align*}
where 
$$C_{Q'} =  \frac{\check{F}(0)L(1,f\otimes g)L(1,\text{sym}^2f)L(1,\text{sym}^2g)Z_{Q'}^{*}(0,0)}{2\pi^2}.$$
By Stirling's formula, the remaining integral on the line $\text{Re}(u)=-1/5$ or on the line $\text{Re}(v)=-1/5$ is  $\ll M^{-1/5+\epsilon}$ while the double integral on the lines $\text{Re}(u)=-1/5$ and $\text{Re}(v)=-1/5$ is $\ll M^{-2/5+\epsilon}$. Let
\begin{align} \label{eqn:residue}
Z^*(u,v) = Z_1^*(u,v) - i^{\kappa_1}\eta_fZ_{q_1}^*(u,v) - i^{\kappa_2}\eta_gZ_{q_2}^*(u,v) + i^{\kappa_1+\kappa_2}\eta_f\eta_gZ_{Q}^*(u,v),
\end{align}
we conclude
\begin{align*}
\text{I}(f,g) = C_0  X (\log M)^2 + C_1 X\log M + O(X^{4/5+\epsilon})
\end{align*}
with
\begin{align*}
C_0 & =   \frac{\check{F}(0)L(1,\text{sym}^2f)L(1,f\otimes g)L(1,\text{sym}^2g)Z^{*}(0,0)}{2\pi^2}
\end{align*}
and
\begin{align} \label{eqn:C1}
C_1& =C_0 \Bigg(\gamma_f'(0)+ \gamma_g'(0)+ 2\frac{L'}{L}(1,\text{sym}^2f) + 2\frac{L}{L}'(1,f\otimes g) +2\frac{L'}{L}(1,\text{sym}^2f) + \frac{\frac{d}{du}Z^*(u,0)\vert_{u=0}}{Z^*(0,0)}+\frac{\frac{d}{dv}Z^*(0,v)\vert_{v=0}}{Z^*(0,0)}   \Bigg).
\end{align}
\vspace{0.1cm}
\subsubsection{(Off-diagonal terms $k\neq 0$)} Now consider
\begin{align*}
& T_{\text{off-diagonal}} \\
&  = \frac{X}{16} \sum_{\substack{(a,2Q)=1\\ a\leq Y}} \frac{\mu(a)}{a^2} \sum_{b\mid Q} \frac{\mu(b)}{b} \sum_{k\neq 0} (-1)^k \sumtwo_{\substack{(n_1n_2,2a)=1}} \frac{\lambda_f(n_1)\lambda_g(n_2)\chi_b(n_1n_2Q')}{\sqrt{n_1n_2}} W_f\left(\frac{n_1}{M}\right)W_g\left(\frac{n_2}{M}\right) \frac{G_k(n_1n_2Q')}{n_1n_2Q'}\check{F}\left(\frac{kX}{16a^2bn_1n_2Q'}\right).
\end{align*} 
From Lemma \ref{lem:GaussSum}, we have $G_{4k}(n)= G_{k}(n)$ for odd $n$. To verify that, it suffices to check that the equality holds for prime powers
\begin{align*}
G_{4k}(p_i^{m_i}) = G_{k}(p_i^{m_i})
\end{align*}
for $p_i>2$. Since $4\nmid p_i$, the equality is true if and only if $p_i^{\alpha} \| 4k$ and $p_i^{\alpha}\| k$. 
\begin{rem}
There is a similar relationship between $G_{2k}(p_i^{m_i})$ and $G_k(p_i^{m_i})$, but there appears an extra term $\chi_2(p_i)$ in the case when $\beta = \alpha+1$ is odd. That is, we will have $G_{2k}(p_i^{m_i}) = \chi_2(p_i) G_k(p_i^{m_i})$ instead. Therefore, it is more convenient to substitute $4k$ instead of $2k$ since $\chi_{4}(n)=1$.
\end{rem}
Thus, we rewrite the sums
\begin{align*}
 \sumthree_{k\neq 0 ,(n_1n_2,2a)=1}=\quad \sum_{\ell \geq 1}T_{\ell}(a,Q') - T_0(a,Q').
\end{align*}
For $\ell\geq 0$, we define
\begin{align*}
T_{\ell}(a,Q') := \sum_{(k,2)=1}\sumtwo_{\substack{(n_1n_2,2a)=1}} \frac{\lambda_f(n_1)\lambda_g(n_2)\chi_b(n_1n_2Q')}{\sqrt{n_1n_2}} W_f\left(\frac{n_1}{M}\right)W_g\left(\frac{n_2}{M}\right) \frac{G_{2^{\delta(\ell)}k}(n_1n_2Q')}{n_1n_2Q'}\check{F}\left(\frac{2^{\ell}kX}{16a^2bn_1n_2Q'}\right),
\end{align*}
where $\delta(\ell)=1$ if $2 \nmid \ell$ and $\delta(\ell)=0$ otherwise. We claim that
\begin{align*}
T_{\ell}(a,Q') & \ll \frac{a^2b M (\log M)^4}{2^{\ell} X},
\end{align*}
which then implies that $T_{\text{off-diagonal}} \ll X/ (\log X)^{10}$. It suffices to bound $T_0(a,Q')$ since similar computations can be used to bound $T_{\ell}(a,Q')$ for any $\ell\geq 1$ with some minor adjustments.
\begin{lemma} \label{lem:T0}
\begin{align*}
T_0(a,Q') \ll \frac{a^2b M (\log M)^4}{X}.
\end{align*}
\end{lemma}
\begin{proof}
For $k$ odd, we write $k=k_1k_2^2$ with $k_1$ squarefree and $k_2$ positive odd integer ($k_1$ can be negative). For $\text{Re}(\alpha),\text{Re}(\beta)=1/2$ and $\mathrm{Re}(\rho)>0$, we define
\begin{align} \label{eqn:T1}
T(k_1,\alpha,\beta,Q';\rho) := \sum_{\substack{k_2\geq 1\\ (k_2, 2)=1}} \sumtwo_{\substack{(n_1n_2,2a)=1}} \frac{\lambda_f(n_1)\lambda_g(n_2)\chi_b(n_1n_2Q')}{n_1^{\alpha}n_2^{\beta}} \frac{G_{k_1k_2^2}(n_1n_2Q')}{n_1n_2Q'} V\bfrac{n_1}{N_1}V\bfrac{n_2}{N_2}  \check{F}_{\rho}\bfrac{k_1k_2^2X}{16a^2bn_1n_2Q'},
\end{align}
where
\begin{align} \label{eqn:V}
V(x) = G(x/2)+G(x)+G(2x)
\end{align}  
satisfies that $V(x) =1$ for $x\in [1/2,3]$ and 
\begin{align*}
\check{F}_{\rho}(y) =\frac{1}{2\pi i}\int_{(1/2)} \tilde{F}(1+\rho-s)\Gamma(s)(\cos+\mathrm{sgn}(y)\sin)\bfrac{\pi s}{2}(2\pi |y|)^{-s}\; ds
\end{align*}
is similar to the Fourier type transform $\check{F}$ defined in (\ref{eqn:fouriertype}) but with an extra parameter $\rho$ inside $\tilde{F}$.
\begin{rem}
For convenience, we will omit $\rho$ when $\rho=0$ so that $T(k_1,\alpha,\beta,Q';\rho)$ and $ \check{F}_{\rho}(y)$ will be replaced by $T(k_1,\alpha,\beta,Q')$ and $\check{F}(y)$ respectively.
\end{rem}
Following (\ref{eqn:cutofffunction}), we insert the smooth partition of unity and get
\begin{align*}
 T_{0}(a,Q')  
& =  \sumd_{N_1,N_2} \left(\frac{1}{2\pi i}\right)^2 \int_{(3)}\int_{(3)} \gamma_f(u)\gamma_g(v)\frac{M^u}{u^2}\frac{M^v}{v^2} \\
& \times \sumstar_{(k_1,2)=1} \sum_{\substack{k_2\geq 1\\ (k_2, 2)=1}}  \sumtwo_{\substack{ (n_1n_2,2a)=1}} \frac{\lambda_f(n_1)\lambda_g(n_2)\chi_b(n_1n_2Q')}{n_1^{1/2+u}n_2^{1/2+v}} \frac{G_{k_1k_2^2}(n_1n_2Q')}{n_1n_2Q'} G\bfrac{n_1}{N_1}G\bfrac{n_2}{N_2} \check{F}\left(\frac{k_1k_2^2X}{16a^2bn_1n_2Q'}\right) \;dudv.
\end{align*}
Next insert functions $V$ over $n_i$ sums, apply Mellin inversion for $G$ and do a change of variables will give us
\begin{align*}
T_0(a,Q') 
& =  \sumd_{N_1,N_2} \left(\frac{1}{2\pi i}\right)^4\int_{(3)}\int_{(3)} \int_{(0)}\int_{(0)}  \gamma_f(u)\gamma_g(v)\frac{\bfrac{M}{N_1}^u}{u^2}\frac{\bfrac{M}{N_2}^v}{v^2}  \tilde{G}(w_1- u)\tilde{G}(w_2-v)  N_1^{w_1}N_2^{w_2}\\
& \times \sumstar_{(k_1,2)=1} T(k_1,1/2+w_1,1/2+w_2,Q') \; dw_1dw_2dudv \\
& = T_{(N_1\leq M ,N_2\leq M)} + T_{(N_1\leq M, M<N_2)} + T_{(M<N_1,N_2\leq M)} +T_{(M<N_1,M<N_2)},
\end{align*}
where we split each $N_i$ sum in two ranges. For  $T_{(N_1,N_2\leq M)}$, move $\text{Re}(u),\text{Re}(v)=3$ to $\text{Re}(u),\text{Re}(v)=-1/5$ and there are poles of order $2$ at $u,v=0$. Then Cauchy's Theorem implies
\begin{align*}
T_{(N_1\leq M,N_2\leq M)} \ll \sumd_{\substack{N_1\leq M \\ N_2\leq M}} \log\bfrac{M}{N_1}\log\bfrac{M}{N_2} \int_{-\infty}^{\infty} \int_{-\infty}^{\infty}  \frac{1}{(1+|t_1|)^{20}} \frac{1}{(1+|t_2|)^{20}} \left|\sumstar_{(k_1,2)=1} T(k_1,1/2+it_1,1/2+it_2,Q')\right| dt_1dt_2.
\end{align*}

Next, we want to bound $T_{(N_1\leq M, M<N_2)}+ T_{(M<N_1,N_2\leq M)}$.  Without loss of generality, it suffices to analyze $T_{(N_1\leq M, M<N_2)}$. Move $\text{Re}(u)= 3$ to $\text{Re}(u)=-1/5$ and $\text{Re}(v) = 3$ to $\text{Re}(v)=6$, we encounter a pole of order $2$ at $u=0$. Then Cauchy's Theorem implies
\begin{align*}
T_{(N_1\leq M, M<N_2)} & \ll \sumd_{\substack{N_1\leq M \\ M< N_2}} \log\bfrac{M}{N_1}\bfrac{M}{N_2}^6 \int_{-\infty}^{\infty} \int_{-\infty}^{\infty} \int_{-\infty}^{\infty}  \frac{1}{(1+|t_1|)^{20}} \frac{1}{(1+|-6+i(t_2-t_4)|)^{20}} \frac{1}{(1+|6+it_4|)^{20}}  \\
& \times \left|\sumstar_{(k_1,2)=1} T(k_1,1/2+it_1,1/2+it_2,Q')\right|  \; dt_1dt_2dt_4.
\end{align*}
For $T_{(M<N_1, M<N_2)}$, move $\text{Re}(u),\text{Re}(v)=3$ to $\text{Re}(u),\text{Re}(v) = 6$ and then
\begin{align*}
 T_{(M<N_1,M<N_2)} & \ll \sumd_{\substack{ M <N_1\\ M< N_2}} \bfrac{M}{N_1}^6 \bfrac{M}{N_2}^6\int_{-\infty}^{\infty} \int_{-\infty}^{\infty} \int_{-\infty}^{\infty} \int_{-\infty}^{\infty} \frac{1}{(1+|-6+i(t_1-t_3)|)^{20}}  \frac{1}{(1+|6+it_3|)^{20}} \\
& \times  \frac{1}{(1+|-6+i(t_2-t_4)|)^{20}}   \frac{1}{(1+|6+it_4|)^{20}}  \left|\sumstar_{(k_1,2)=1} T(k_1,1/2+it_1,1/2+it_2,Q')\right|  \; dt_1dt_2dt_3dt_4.
\end{align*}
Combine all of the contributions above, we derive
\begin{align} \label{eqn:Tzerobound}
 T_{0}(a,Q') \ll \sumd_{\substack{N_1\leq M \\ N_2\leq M}} \log\bfrac{M}{N_1}\log\bfrac{M}{N_2} \int_{-\infty}^{\infty} \int_{-\infty}^{\infty}   \frac{1}{(1+|t_1|)^{20}} \frac{1}{(1+|t_2|)^{20}} \left|\sumstar_{(k_1,2)=1} T(k_1,1/2+it_1,1/2+it_2,Q')\right| \; dt_1dt_2 
\end{align}
\begin{align*}
\quad \quad & + \sumd_{\substack{N_1\leq M \\ M< N_2}} \log\bfrac{M}{N_1}\bfrac{M}{N_2}^6 \int_{-\infty}^{\infty} \int_{-\infty}^{\infty} \int_{-\infty}^{\infty} \frac{1}{(1+|t_1|)^{20}} \frac{1}{(1+|-6+i(t_2-t_4)|)^{20}}   \frac{1}{(1+|6+it_4|)^{20}}  \\
& \times  \left|\sumstar_{(k_1,2)=1} T(k_1,1/2+it_1,1/2+it_2,Q')\right|  \; dt_1dt_2dt_4 \\
& + \sumd_{\substack{ M <N_1\\ M< N_2}} \bfrac{M}{N_1}^6 \bfrac{M}{N_2}^6\int_{-\infty}^{\infty} \int_{-\infty}^{\infty} \int_{-\infty}^{\infty} \int_{-\infty}^{\infty} \frac{1}{(1+|-6+i(t_1-t_3)|)^{20}}   \frac{1}{(1+|6+it_3|)^{20}}  \\
& \times  \frac{1}{(1+|-6+i(t_2-t_4)|)^{20}}   \frac{1}{(1+|6+it_4|)^{20}} \left|\sumstar_{(k_1,2)=1} T(k_1,1/2+it_1,1/2+it_2,Q')\right|  \; dt_1dt_2dt_3dt_4. 
\end{align*}
\subsection{Estimation of $k_1$-sum} \label{sec:1}
Recall
\begin{align*}
T(k_1,1/2+it_1,1/2+it_2,Q') & = \sum_{\substack{k_2\geq 1\\ (k_2, 2)=1}} \sumtwo_{\substack{(n_1n_2,2a)=1}} \frac{\lambda_f(n_1)\lambda_g(n_2)\chi_b(n_1n_2Q')}{n_1^{1/2+it_1}n_2^{1/2+it_2}} \frac{G_{k_1k_2^2}(n_1n_2Q')}{n_1n_2Q'} V\bfrac{n_1}{N_1}V\bfrac{n_2}{N_2}  \check{F}\bfrac{k_1k_2^2X}{16a^2bn_1n_2Q'}.
\end{align*}
By (\ref{eqn:fouriertype}) and apply Mellin inversion for $V$, we have
\begin{align*}
& T(k_1,1/2+it_1,1/2+it_2,Q') \\
&= \bfrac{1}{2\pi i}^3 \int_{(1/2)} \int_{(\epsilon)}\int_{(\epsilon)} \tilde{F}(1-s)\Gamma(s) (\cos+\text{sgn}(k_1)\sin)\bfrac{\pi s}{2} \bfrac{\pi |k_1|X}{8a^2bQ'}^{-s} \tilde{V}(u)\tilde{V}(v) N_1^uN_2^v\\
& \times \sum_{\substack{k_2\geq 1 \\(k_2,2)=1}} \sumtwo_{\substack{(n_1n_2,2a)=1}} \frac{\lambda_f(n_1)\lambda_g(n_2)\chi_b(n_1n_2Q')}{n_1^{1/2+it_1+u-s}n_2^{1/2+it_2+v-s}k_2^{2s}} \frac{G_{k_1k_2^2}(n_1n_2Q')}{n_1n_2Q'}  \;dudvds \\
 &= \bfrac{1}{2\pi i}^3 \int_{(1/2)} \int_{(\epsilon)}\int_{(\epsilon)} \tilde{F}(1-s)\Gamma(s) (\cos+\text{sgn}(k_1)\sin)\bfrac{\pi s}{2} \bfrac{\pi |k_1|X}{8a^2bQ'}^{-s} \tilde{V}(u)\tilde{V}(v) N_1^uN_2^v\\
& \times Z(1/2+it_1+u-s,1/2+it_2+v-s,s;a,k_1,Q',\chi_b)  \;dudvds.
\end{align*}
Apply Lemma \ref{lem:DirichletSeries} to $Z(1/2+it_1+u-s,1/2+it_2+v-s,s;a,k_1,Q',\chi_b)$, then we write
\begin{align*}
& T(k_1,1/2+it_1,1/2+it_2,Q') \\
& = \bfrac{1}{2\pi i}^3 \int_{(1/2)}\int_{(\epsilon)}\int_{(\epsilon)} \tilde{F}(1-s)\Gamma(s)(\cos+\text{sgn}(k_1)\sin)\bfrac{\pi s}{2} \bfrac{8a^2bQ'}{\pi |k_1|X}^s \tilde{V}(u)\tilde{V}(v) N_1^uN_2^v\\
 &\times \sumthree_{r_1,r_2,r_3} \frac{C(r_1,r_2,r_3)}{r_1^{1/2+it_1+u-s}r_2^{1/2+it_2+v-s}r_3^{2s}} \sum_{n_1}\frac{\lambda_{f_{b}}(n_1)\chi_{m(k_1)}(n_1)}{n_1^{1+it_1+ u-s}} \sum_{n_2}\frac{\lambda_{g_{b}}(n_2)\chi_{m(k_1)}(n_2)}{n_2^{1+it_2+v-s}} \; dudvds \\
 & = \frac{1}{2\pi i} \int_{(1/2)} \tilde{F}(1-s)\Gamma(s)(\cos+\text{sgn}(k_1)\sin)\bfrac{\pi s}{2} \bfrac{8a^2bQ'}{\pi |k_1|X}^s \sumthree_{r_1,r_2,r_3} \frac{C(r_1,r_2,r_3)}{r_1^{1/2+it_1-s}r_2^{1/2+it_2-s}r_3^{2s}} \\
 &\times \sum_{n_1}\frac{\lambda_{f_{b}}(n_1)\chi_{m(k_1)}(n_1)}{n_1^{1+it_1-s}} \sum_{n_2}\frac{\lambda_{g_{b}}(n_2)\chi_{m(k_1)}(n_2)}{n_2^{1+it_2-s}}V\bfrac{r_1n_1}{N_1}V\bfrac{r_2n_2}{N_2} \; ds.
\end{align*} 
Let 
\begin{align*}
V_1(x) =G(4x)+G(2x)+G(x)+G(x/2)+G(x/4)
\end{align*}
satisfies that $V_1(x)=1$ for $x\in [1,6]$. We insert functions $V_1$ over $n_i$ sums and apply Mellin inversion for $V$.
\begin{align*}
& T(k_1,1/2+it_1,1/2+it_2,Q') \\
& = \bfrac{1}{2\pi i}^3 \int_{(1/2)}\int_{(\epsilon)}\int_{(\epsilon)} \tilde{F}(1-s)\Gamma(s)(\cos+\text{sgn}(k_1)\sin)\bfrac{\pi s}{2} \bfrac{8a^2bQ'}{\pi |k_1|X}^s \tilde{V}(u)\tilde{V}(v) N_1^uN_2^v\\
 &\times \sumthree_{r_1,r_2,r_3} \frac{C(r_1,r_2,r_3)}{r_1^{1/2+it_1+u-s}r_2^{1/2+it_2+v-s}r_3^{2s}} \sum_{n_1}\frac{\lambda_{f_{b}}(n_1)\chi_{m(k_1)}(n_1)}{n_1^{1+it_1+ u-s}} \sum_{n_2}\frac{\lambda_{g_{b}}(n_2)\chi_{m(k_1)}(n_2)}{n_2^{1+it_2+v-s}}V_1\bfrac{r_1n_1}{N_1}V_1\bfrac{r_2n_2}{N_2} \; dudvds.
 \end{align*}
 Do a change of variables and we insert the smooth partition of unity over $r_1,r_2$ sums
 \begin{align*}
 & T(k_1,1/2+it_1,1/2+it_2,Q') \\
& = \bfrac{1}{2\pi i}^3 \int_{(1/2)}\int_{(\epsilon)}\int_{(\epsilon)} \tilde{F}(1-s)\Gamma(s)(\cos+\text{sgn}(k_1)\sin)\bfrac{\pi s}{2} \bfrac{8a^2bQ'N_1N_2}{\pi |k_1|X}^s \tilde{V}(u+s)\tilde{V}(v+s) N_1^{u}N_2^{v}\\
 &\times \sumd_{R_1,R_2} \sumthree_{r_1,r_2,r_3} \frac{C(r_1,r_2,r_3)}{r_1^{1/2+it_1+u}r_2^{1/2+it_2+ v}r_3^{2s}}G\bfrac{r_1}{R_1}G\bfrac{r_2}{R_2} \\
 & \times  \sum_{n_1}\frac{\lambda_{f_{b}}(n_1)\chi_{m(k_1)}(n_1)}{n_1^{1+it_1+u}} \sum_{n_2}\frac{\lambda_{g_{b}}(n_2)\chi_{m(k_1)}(n_2)}{n_2^{1+it_2+v}}V_1\bfrac{r_1n_1}{N_1}V_1\bfrac{r_2n_2}{N_2}\; dudvds.
 \end{align*}
We bound the sums over $r_i$ with the following lemma:
 \begin{lemma}\label{lem:DirichletSeriesCoefficient}
 Suppose $\mathrm{Re}(s)\geq 3/5$ and write $u=-1/2+i\mu$, $v=-1/2+i\nu$ for real $\mu,\nu$. We have that
 \[
 \left|\sumthree_{r_1,r_2,r_3}\frac{C(r_1,r_2,r_3)}{r_1^{1/2+it_1+u}r_2^{1/2+it_2+v}r_3^{2s}}G\bfrac{r_1}{R_1}G\bfrac{r_2}{R_2}\right|\ll (1+|t_1|)^{1/5}(1+|t_2|)^{1/5}(1+|\mu|)(1+|\nu|)\exp(-c_1\sqrt{\log(R_1R_2)}),
 \]
 where the implied constant and $c_1>0$ may depend only on $f$ and $g$.
 \end{lemma}
The lemma above and its proof are nearly identical to \cite[Lemma 5.7]{Li} so we omit the proof here. The only difference is that we need to use the fact that $L(s,f\otimes g)$ has similar zero-free region with $L(s,\text{sym}^2f)$, which we will refer the reader to \cite[Theorem 5.44]{IK} for more details.

We first move $\text{Re}(u),\text{Re}(v) = \epsilon$ to $\text{Re}(u) ,\text{Re}(v)= -1/2$. Next, we split the $k_1$ sum in two ranges: $|k_1|>J$ and $|k_1|\leq J$ with $J=\frac{8a^2bQ'N_1N_2}{\pi X}$.
For $|k_1|>J$, move $\text{Re}(s)=1/2$ to $\text{Re}(s) = 6/5$ and then
\begin{align*} 
& \left|\; \sumstar_{\substack{(k_1,2)=1\\ |k_1|>J}} T(k_1,1/2+it_1,1/2+it_2,Q')\right|\\
& \ll \int_{-\infty}^{\infty} \int_{-\infty}^{\infty} \int_{-\infty}^{\infty} \frac{1}{(1+|1-(6/5+it)|)^{20}} \frac{1}{(1+|(-1/2+i\mu) + (6/5+it)|)^{20}} \frac{1}{(1+|(-1/2+i\nu)+(6/5+it)|)^{20}}   \\
 &\times  \bfrac{8a^2bQ'N_1N_2}{\pi X}^{6/5} N_1^{-1/2}N_2^{-1/2}  \sumd_{R_1,R_2} (1+|t_1|)^{1/5}(1+|t_2|)^{1/5}(1+|\mu|)(1+|\nu|)\exp(-c_1\sqrt{\log(R_1R_2)}) \\
 & \times \sumstar_{\substack{(k_1,2)=1\\ |k_1|>J}}  \frac{1}{|k_1|^{6/5}} \left|\sum_{n_1}\frac{\lambda_{f_{b}}(n_1)\chi_{m(k_1)}(n_1)}{n_1^{1/2+i(t_1+\mu)}} \sum_{n_2}\frac{\lambda_{g_{b}}(n_2)\chi_{m(k_1)}(n_2)}{n_2^{1/2+i(t_2+\nu)}}V_1\bfrac{R_1n_1}{N_1}V_1\bfrac{R_2n_2}{N_2}\;\right| d\mu d\nu dt,
 \end{align*}
where we have used Lemma \ref{lem:DirichletSeriesCoefficient} and the standard bounds 
\begin{align}  \label{eqn:standardbounds}
\Gamma(s)\cos\bfrac{\pi s}{2} & \ll |s|^{\text{Re}(s)-1/2}\quad;\quad \tilde{G}(s), \tilde{F}(s) \ll \frac{1}{(1+|s|)^{20}}.
\end{align}
 Apply Cauchy-Schwarz inequality and use Lemma \ref{lem:NonDyadicCut2}, we deduce
\begin{align*} 
&\sumstar_{\substack{(k_1,2)=1\\ |k_1|>J}}  \frac{1}{|k_1|^{6/5}} \left|\sum_{n_1}\frac{\lambda_{f_{b}}(n_1)\chi_{m(k_1)}(n_1)}{n_1^{1/2+i(t_1+\mu)}} \sum_{n_2}\frac{\lambda_{g_{b}}(n_2)\chi_{m(k_1)}(n_2)}{n_2^{1/2+i(t_2+\nu)}}V_1\bfrac{R_1n_1}{N_1}V_1\bfrac{R_2n_2}{N_2}\;\right| \\ 
&  \ll J^{-1/5} (1+|t_1+\mu|)^3  (1+|t_2+\nu|)^3 \log(2+|t_1+\mu|) \log(2+|t_2+\nu|).
\end{align*}

When $|k_1|\leq J$, we move $\text{Re}(s) = 1/2$ to $\text{Re}(s) = 3/5$.
\begin{align*} 
& \left|\; \sumstar_{\substack{(k_1,2)=1\\ |k_1|\leq J}} T(k_1,1/2+it_1,1/2+it_2,Q')\right|\\
& \ll \int_{-\infty}^{\infty} \int_{-\infty}^{\infty} \int_{-\infty}^{\infty} \frac{1}{(1+|1-(3/5+it)|)^{20}} \frac{1}{(1+|(-1/2+i\mu) + (3/5+it)|)^{20}} \frac{1}{(1+|(-1/2+i\nu)+(3/5+it)|)^{20}}   \\
 &\times  \bfrac{8a^2bQ'N_1N_2}{\pi X}^{3/5} N_1^{-1/2}N_2^{-1/2}  \sumd_{R_1,R_2} (1+|t_1|)^{1/5}(1+|t_2|)^{1/5}(1+|\mu|)(1+|\nu|)\exp(-c_1\sqrt{\log(R_1R_2)}) \\
 & \times \sumstar_{\substack{(k_1,2)=1\\ |k_1|\leq J}}  \frac{1}{|k_1|^{3/5}} \left|\sum_{n_1}\frac{\lambda_{f_{b}}(n_1)\chi_{m(k_1)}(n_1)}{n_1^{1/2+i(t_1+\mu)}} \sum_{n_2}\frac{\lambda_{g_{b}}(n_2)\chi_{m(k_1)}(n_2)}{n_2^{1/2+i(t_2+\nu)}}V_1\bfrac{R_1n_1}{N_1}V_1\bfrac{R_2n_2}{N_2}\;\right| d\mu d\nu dt.
 \end{align*}
Then apply Cauchy-Schwarz inequality and use Lemma \ref{lem:NonDyadicCut2}, we deduce
\begin{align*}
&  \sumstar_{\substack{(k_1,2)=1\\ |k_1|\leq J}}  \frac{1}{|k_1|^{3/5}} \left|\sum_{n_1}\frac{\lambda_{f_{b}}(n_1)\chi_{m(k_1)}(n_1)}{n_1^{1/2+i(t_1+\mu)}} \sum_{n_2}\frac{\lambda_{g_{b}}(n_2)\chi_{m(k_1)}(n_2)}{n_2^{1/2+i(t_2+\nu)}}V_1\bfrac{R_1n_1}{N_1}V_1\bfrac{R_2n_2}{N_2}\;\right| \\
& \ll J^{2/5} (1+|t_1+\mu|)^3  (1+|t_2+\nu|)^3 \log(2+|t_1+\mu|) \log(2+|t_2+\nu|).
\end{align*}
So for any fixed $N_1$ and $N_2$, we have the estimation
\begin{align} \label{eqn:fixed}
\sumstar_{(k_1,2)=1} T(k_1,1/2+it_1,1/2+it_2,Q')
\end{align}
\begin{align*}
& \ll \int_{-\infty}^{\infty} \int_{-\infty}^{\infty} \int_{-\infty}^{\infty} \frac{1}{(1+|t|)^{20}} \frac{1}{(1+|\mu+t|)^{20}} \frac{1}{(1+|\nu + t|)^{20}} \sumd_{R_1,R_2} (1+|t_1|)^{1/5} (1+|t_2|)^{1/5} (1+|\mu|)(1+|\nu|) \exp(-c_1\sqrt{\log(R_1R_2)})  \\
& \times  \frac{a^2b\sqrt{N_1N_2}}{X}  (1+|t_1+\mu|)^3  (1+|t_2+\nu|)^3  \log(2+|t_1+\mu|) \log(2+|t_2+\nu|)  \; d\mu d\nu dt.
\end{align*} 
Combine (\ref{eqn:Tzerobound}) and (\ref{eqn:fixed}), we see that $T_0(a,Q')$ is 
\begin{align*}
&  \ll \frac{a^2b}{X} \Bigg(\; \sumd_{\substack{N_1\leq M \\ N_2\leq M}} \log\bfrac{M}{N_1}\log\bfrac{M}{N_2} \sqrt{N_1N_2} + \sumd_{\substack{N_1\leq M \\ M<N_2}} \log\bfrac{M}{N_1}\bfrac{M}{N_2}^6 \sqrt{N_1N_2} +  \sumd_{\substack{M<N_1 \\ M<N_2}} \bfrac{M}{N_1}^6 \bfrac{M}{N_2}^6 \sqrt{N_1N_2}   \Bigg) \\
& \ll \frac{a^2b}{X} \Bigg(M (\log M)^4 + M (\log M)^2 + M \Bigg) \ll \frac{a^2b M (\log M)^4 }{X},
\end{align*}
where we have used
\begin{align*}
\sumd_{R_1,R_2} \exp(-c_1\sqrt{\log(R_1R_2))} & = \sum_{n=0}^{\infty} \sum_{m=0}^{\infty} \exp(-c_1\sqrt{(n+m)\log2}) = \sum_{h = 0}^{\infty} \frac{h+1}{e^{c_1\sqrt{h\log2}}} \leq \sum_{h=1}^{\infty} \frac{1}{h^2} \ll 1
\end{align*}
and
\begin{align*}
\sumd_{M<N} \bfrac{M}{N}^6 \sqrt{N} \ll M^6 \sumd_{M<N} \bfrac{1}{N}^{11/2} \ll \sqrt{M}.
\end{align*}
\end{proof}

\subsection{Bound of $T_{\ell}(a,Q')$}
For $\ell\geq 1$, there are two cases we need to consider: if $\delta(\ell)=0$, we have
\[
T_{\ell}(a,Q') = \sum_{(k,2)=1}\sumtwo_{\substack{(n_1n_2,2a)=1}} \frac{\lambda_f(n_1)\lambda_g(n_2)\chi_{b}(n_1n_2Q')}{\sqrt{n_1n_2}} W_f\left(\frac{n_1}{M}\right)W_g\left(\frac{n_2}{M}\right) \frac{G_{k}(n_1n_2Q')}{n_1n_2Q'}\check{F}\left(\frac{2^{\ell}kX}{16a^2bn_1n_2Q'}\right).
\]
Write $k = k_1k_2^2$ with $k_1$ odd squarefree integer and $k_2$ positive odd integer, then similar computations in the proof of Lemma \ref{lem:T0} will give us
\begin{align} \label{eqn:delta0}
T_{\ell}(a,Q') \ll\frac{a^2b M (\log M)^4 }{2^{\ell} X}.
\end{align}

If $\delta(\ell)=1$, we instead have
\[
T_{\ell}(a,Q') = \sum_{(k,2)=1}\sumtwo_{\substack{(n_1n_2,2a)=1}} \frac{\lambda_f(n_1)\lambda_g(n_2)\chi_{b}(n_1n_2Q')}{\sqrt{n_1n_2}} W_f\left(\frac{n_1}{M}\right)W_g\left(\frac{n_2}{M}\right) \frac{G_{2k}(n_1n_2Q')}{n_1n_2Q'}\check{F}\left(\frac{2^{\ell}kX}{16a^2bn_1n_2Q'}\right).
\]
We write $2k = k_1k_2^2$ with $k_1$ even squarefree integer and $k_2$ positive odd integer. Following (\ref{eqn:fouriertype}) and (\ref{eqn:cutofffunction}), we then apply Mellin inversion for $V$ 
and encounter the series in the form of
\[
Z_1(\alpha,\beta,\gamma;a,k_1,Q'\chi_b) := \sum_{\substack{k_2\geq 1 \\(k_2,2)=1}} \sumtwo_{\substack{(n_1n_2,2a)=1}} \frac{\lambda_f(n_1)\lambda_g(n_2)\chi_{b}(n_1n_2Q')}{n_1^{\alpha}n_2^{\beta}k_2^{2\gamma}} \frac{G_{k_1k_2^2}(n_1n_2Q')}{n_1n_2Q'}
\]
with $k_1$ even.
\begin{lemma} \label{lem:DirichletSeries2}
Let $f_{2b} :=f\otimes\chi_{2b}, g_{2b} := g\otimes\chi_{2b}$ be the twisted newforms. Let $k_1 = 2k_1'$ and 
\[
m(k_1') =
\begin{cases}
k_1', & \text{if $k_1'\equiv 1\pmod{4}$}, \\
4k_1', & \text{if $k_1'\equiv 3\pmod{4}$}. \\
\end{cases}
\]
Then
\[
Z_1(\alpha,\beta,\gamma; a,k_1,Q',\chi_b) = L(1/2+\alpha,f_{2b}\otimes \chi_{m(k_1')}) L(1/2+\beta,g_{2b}\otimes \chi_{m(k_1')})Y_1(\alpha,\beta,\gamma;a,k_1,Q',\chi_b)
\]
with
\[
Y_1(\alpha,\beta, \gamma;a,k_1,Q',\chi_b) = \frac{Z_1^{*}(\alpha,\beta,\gamma;a,k_1,Q',\chi_b)}{L(1+2\alpha,\text{sym}^2f) L(1+2\beta,\text{sym}^2g) L(1+\alpha+\beta,f\otimes g)},
\]
where $Z_1^*(\alpha,\beta,\gamma;a,k_1,Q',\chi_b)$ is analytic in the region $\mathrm{Re}(\alpha),\mathrm{Re}(\beta)\geq -\delta/2$ and $\mathrm{Re}(\gamma)\geq 1/2+\delta$ for any $0<\delta<1/3$. Moreover, in the same region, $Z_1^*(\alpha,\beta,\gamma; a,k_1,Q',\chi_b) \ll\tau(a)$ where the implied constant may depend on $\delta, f$ and $g$. Furthermore, we can represent
\[
Y_1(\alpha,\beta,\gamma;a,k_1,Q',\chi_b) = \sumthree_{r_1,r_2,r_3} \frac{C_1(r_1,r_2,r_3)}{r_1^{\alpha}r_2^{\beta}r_3^{2\gamma}}
\]
for some coefficients $C_1(r_1,r_2,r_3)$.
\end{lemma}
\begin{proof} The proof is almost identical to the one for Lemma \ref{lem:DirichletSeries}. Except for the case when $p\nmid ak_1$, the corresponding Euler factor is
\[
\sumthree_{k_2,n_1,n_2\geq 0} \frac{\lambda_f(p^{n_1}) \lambda_g(p^{n_2}) \chi_b(p^{n_1+n_2+r_p}) }{p^{n_1\alpha+n_2\beta+2k_2\gamma}} \frac{G_{k_1p^{2k_2}}(p^{n_1+n_2+r_p})}{p^{n_1+n_2+ r_p}} =
\begin{cases}
1+\frac{\chi_{2b}(p)}{p^{1/2}}\left(\frac{\lambda_f(p)\chi_{k_1'}(p)}{p^{\alpha}}+\frac{\lambda_g(p)\chi_{k_1'}(p)}{p^{\beta}}\right) + O\bfrac{1}{p^{1+2\delta}}, & \text{if $r_p= 0$}, \\
\frac{\chi_{2b}(p)\chi_{k_1'}(p)}{p^{1/2}} + O\left(\frac{1}{p^{1+2\delta}}\right), & \text{if $r_p = 1$}, \\
O\left(\frac{1}{p^{1+2\delta}}\right), & \text{if $r_p > 1$}.
\end{cases}
\]
\end{proof}
Follow by the computations in the proof of Lemma \ref{lem:T0} and Lemma \ref{lem:DirichletSeries2}, we also derive
\begin{align} \label{eqn:delta1}
T_{\ell}(a,Q') \ll\frac{a^2b M (\log M)^4 }{2^{\ell} X}.
\end{align}
when $\delta(\ell) = 1$. Therefore, we conclude that
\begin{align*}
T_{\text{off-diagonal}}   \ll  X \sum_{\substack{(a,2Q)=1 \\ a\leq Y}} \frac{1}{a^2} \sum_{b\mid Q} \frac{1}{b} \sum_{\ell\geq 0} \frac{a^2b M (\log M)^4 }{2^{\ell} X} \ll YM(\log M)^4  \ll \frac{X}{(\log X)^{10}}
\end{align*}
by choosing $Y = (\log X)^{20}$.

\section{Proof of Proposition \ref{prop:Second}}
Let $h_1(n,m,d) = W_f\bfrac{n}{M}W_g(m;|8d|,M)F\bfrac{8d}{X}$ with $W_g(m;|8d|,M) := W_g\bfrac{m}{|8d|}-W_g\bfrac{m}{M}$, then
\begin{align*}
\text{II}(f,g) = S_{f,g}(1;h_1) - i^{\kappa_1}\eta_fS_{f,g}(q_1;h_1)- i^{\kappa_2}\eta_gS_{f,g}(q_2;h_1) +  i^{\kappa_1} i^{\kappa_2}\eta_f\eta_gS_{f,g}(Q;h_1),
\end{align*}
where $S_{f,g}(Q';h)$ is defined as in (\ref{eqn:decomposition}). Apply M\"obius inversion and then we split $S_{f,g}(Q';h_1)$ into two ranges:
\begin{align*}
T_{(a\leq Y)} & = \sum_{\substack{(a,2Q)=1 \\ a\leq Y}} \mu(a) \sum_{b\mid Q} \mu(b) \sumtwo_{(n_1n_2,2a)=1} \frac{\lambda_f(n_1)\lambda_g(n_2)}{\sqrt{n_1n_2}} W_f\bfrac{n_1}{M}W_g(n_2;|8a^2bd|,M)\sum_{(d,2)=1} \chi_{8bd}(n_1n_2Q') F\bfrac{8a^2bd}{X},
\end{align*}
and
\begin{align*}
T_{(a> Y)} & = \sum_{\substack{(a,2Q)=1 \\ a> Y}} \mu(a) \sum_{b\mid Q} \mu(b) \sumtwo_{(n_1n_2,2a)=1} \frac{\lambda_f(n_1)\lambda_g(n_2)}{\sqrt{n_1n_2}} W_f\bfrac{n_1}{M}W_g(n_2;|8a^2bd|,M)\sum_{(d,2)=1} \chi_{8bd}(n_1n_2Q') F\bfrac{8a^2bd}{X}.
\end{align*}
\subsection{Bound of $T_{(a>Y)}$} Apply Cauchy-Schwarz inequality and using the bounds from \cite[Proposition 2 and 3]{KMSS}, we have
\begin{align*}
T_{(a>Y)} \ll \frac{X (\log M)^{5/2+\epsilon}}{Y^{1-\epsilon}}.
\end{align*} 
\subsection{Asymptotic of $T_{(a\leq Y)}$} By (\ref{eqn:cutofffunction}), we write
\begin{align*}
 T_{(a\leq Y)} & = \sum_{\substack{(a,2Q)=1 \\ a\leq Y}} \mu(a) \sum_{b\mid Q} \mu(b) \bfrac{1}{2\pi i}^2 \int_{(3)} \int_{(3)}\gamma_f(u)\gamma_g(v) \frac{M^u}{u^2} \frac{1}{v^2} \sumtwo_{(n_1n_2,2a)=1} \frac{\lambda_f(n_1)\lambda_g(n_2)\chi_b(n_1n_2Q')}{n_1^{1/2+u}n_2^{1/2+v}} \\
&\times \sum_{(d,2)=1} \chi_{8d}(n_1n_2Q') H_v\bfrac{8a^2bd}{X} \; dudv,
\end{align*}
where $H_v(s): = F(s)[(|s|X)^v-M^v]$ is also a smooth function compactly support on $[1/2,2]$. After applying Lemma \ref{lem:PoissonSummation} and by (\ref{eqn:Mobius}), we extract the main term $\mathcal{M}_{Q'}$ and the error term $\mathcal{E}_{Q'}$
\begin{align*}
\mathcal{M}_{Q'} = \frac{X}{2\pi^2}  \bfrac{1}{2\pi i}^2 \int_{(3)} \int_{(3)} \gamma_f(u)\gamma_g(v) \frac{M^u}{u^2} \frac{1}{v^2} \sumtwo_{\substack{(n_1n_2,2a)=1 \\ n_1n_2Q' =\square}} \frac{\lambda_f(n_1)\lambda_g(n_2)}{n_1^{1/2+u}n_2^{1/2+v}} \prod_{p\mid n_1n_2Q} \frac{p}{p+1} \check{H_v}(0)\; dudv,
\end{align*}
and
\begin{align*}
\mathcal{E}_{Q'} \ll \frac{X}{16Y} \bfrac{1}{2\pi i}^2 \int_{(3)} \int_{(3)} \gamma_f(u)\gamma_g(v) \frac{M^u}{u^2} \frac{1}{v^2} \sumtwo_{\substack{(n_1n_2,2a)=1 \\ n_1n_2Q' =\square}} \frac{\lambda_f(n_1)\lambda_g(n_2)}{n_1^{1/2+u}n_2^{1/2+v}} \prod_{p\mid n_1n_2Q} \frac{p}{p+1} \check{H_v}(0)\; dudv \ll \frac{X (\log X)^{11}}{Y}.
\end{align*}
By (\ref{eqn:fouriertype1}) and (\ref{eqn:mainterm}), we further have
\begin{align*}
\mathcal{M}_{Q'}  =  \frac{X}{2\pi^2} \int_{-\infty}^{\infty} F(x) \bfrac{1}{2\pi i}^2  \int_{(3)} \int_{(3)} \gamma_f(u)\gamma_g(v) & \frac{M^u}{u^2} \frac{(|x|X)^v-M^v}{v^2} \\ 
&\times  L(1+u+v,f\otimes g)L(1+2u,\text{sym}^2f)L(1+2v,\text{sym}^2g) Z_{Q'}^*(u,v)  \; dudvdx.
\end{align*}
Move $\text{Re}(u),\text{Re}(v) = 3$ to $\text{Re}(u),\text{Re}(v)=-1/5$, we encounter a pole of order $2$ at $u=0$ and a simple pole at $v=0$. Then Cauchy's Theorem implies that
\begin{align*}
\mathcal{M}_{Q'} = \frac{X\log M\log\bfrac{X}{M} L(1,f\otimes g)L(1,\text{sym}^2f)L(1,\text{sym}^2g)Z_{Q'}^*(0,0)\check{F}(0)}{2\pi ^2} + CX\log M+O(X\log \log X)
\end{align*}
with some explicit constant $C$ that depends on $f,g$ and $F$. We further conclude
\begin{align*}
\text{II}(f,g) = \frac{X\log M\log\bfrac{X}{M} L(1,f\otimes g)L(1,\text{sym}^2f)L(1,\text{sym}^2g)Z^*(0,0)\check{F}(0)}{2\pi ^2} + CX\log M+ O(X\log\log X),
\end{align*}
where $Z^*(u,v)$ is defined as in (\ref{eqn:residue}).
\subsubsection{(Off-diagonal terms $k\neq 0$.)} Now consider
\begin{align*}
T_{\text{off-diagonal}} =\frac{X}{16}\sum_{\substack{(a,2Q)=1 \\ a\leq Y}} \frac{\mu(a)}{a^2} & \sum_{b\mid Q} \frac{\mu(b)}{b} \bfrac{1}{2\pi i}^2  \int_{(3)} \int_{(3)} \gamma_f(u)\gamma_g(v) \frac{M^u}{u^2} \frac{1}{v^2}  \\& \times \sum_{k\neq 0} (-1)^k  \sumtwo_{\substack{(n_1n_2,2a)=1}} \frac{\lambda_f(n_1)\lambda_g(n_2)\chi_{b}(n_1n_2Q')}{n_1^{1/2+u}n_2^{1/2+v}}\frac{G_k(n_1n_2Q')}{n_1n_2Q'}\check{H_v}\bfrac{kX}{16a^2bn_1n_2Q'}\; dudv.
\end{align*}
We first rewrite the sums
\begin{align*}
\sum_{k\neq 0} \sumtwo_{(n_1n_2,2a)=1} = \sum_{\ell\geq 1} T_{\ell}(a,Q') - T_{0}(a,Q')
\end{align*}
with
\begin{align*}
& T_{\ell}(a,Q') \\
&  = \bfrac{1}{2\pi i}^2  \int_{(3)} \int_{(3)} \gamma_f(u)\gamma_g(v) \frac{M^u}{u^2} \frac{1}{v^2} \sum_{(k,2)=1} \sumtwo_{\substack{(n_1n_2,2a)=1}} \frac{\lambda_f(n_1)\lambda_g(n_2)\chi_{b}(n_1n_2Q')}{n_1^{1/2+u}n_2^{1/2+v}}\frac{G_{2^{\delta(\ell)}k}(n_1n_2Q')}{n_1n_2Q'}\check{H_v}\bfrac{2^{\ell}kX}{16a^2bn_1n_2Q'}\; dudv,
\end{align*}
and $\delta(\ell)=1$ if $2\nmid \ell$ and $0$ otherwise. It suffices to bound $T_{0}(a,Q')$. We first insert the smooth partition of unity $G$ and then functions $V$ over $n_i$ sums. Apply Mellin inversion for $G$ and do a change of variables, then follow by (\ref{eqn:fouriertype}) and we have
\begin{align*}
T_{0}(a,Q') &  = \sumd_{N_1,N_2} \bfrac{1}{2\pi i}^5 \int_{(1/2)} \int_{0}^{\infty} F(x)x^{-s} \int_{(3)} \int_{(3)} \int_{(0)} \int_{(0)} \gamma_f(u)\gamma_g(v) \frac{\bfrac{M}{N_1}^u}{u^2}  \frac{\frac{(|x|X)^v-M^v}{N_2^v}}{v^2} \tilde{G}(w_1-u)\tilde{G}(w_2-v) N_1^{w_1}N_2^{w_2}\\ 
& \times  \sumstar_{(k_1,2)=1} \sum_{\substack{ k_2\geq 1 \\ (k_2,2)=1}} \sumtwo_{(n_1n_2,2a)=1} \frac{\lambda_f(n_1)\lambda_g(n_2)\chi_{b}(n_1n_2Q')}{n_1^{1/2+w_1}n_2^{1/2+w_2}} \frac{G_{k_1k_2^2}(n_1n_2Q')}{n_1n_2Q'}V\bfrac{n_1}{N_1}V\bfrac{n_2}{N_2} \Gamma(s) (\cos+\text{sgn}(k_1)\sin)\bfrac{\pi s}{2} \\
&\times \bfrac{\pi |k_1|k_2^2 X}{8a^2bn_1n_2Q'}^{-s} dw_1dw_2dudvdxds \\
& = T_{(N_1\leq M,N_2\leq M)} + T_{(N_1\leq M, M<N_2\leq X)} +T_{(N_1\leq M, X<N_2)} + T_{(M<N_1, N_2\leq M)} +T_{(M<N_1, M<N_2\leq X)} +T_{(M<N_1, X<N_2)},
\end{align*}
where we split $N_1$ sum in two ranges and $N_2$ sum in three ranges. Move $\text{Re}(u),\text{Re}(v)=3$ to $\text{Re}(u),\text{Re}(v)=-1/5$, we have 
\begin{align*}
T_{(N_1\leq M,N_2\leq M)} \ll \sumd_{\substack{N_1\leq M \\N_2\leq M}} \log\bfrac{M}{N_1}\log\bfrac{X}{M}\int_{-\infty}^{\infty} \int_{-\infty}^{\infty} \frac{1}{(1+|t_1|)^{20}} \frac{1}{(1+|t_2|)^{20}} \Bigg|\sumstar_{(k_1,2)=1} T(k_1,1/2+it_1,1/2+it_2,Q')\Bigg| \; dt_1dt_2,
\end{align*}
where $T(k_1,1/2+it_1,1/2+it_2,Q')$ is defined as in (\ref{eqn:T1}).

For $T_{(N_1\leq M, M<N_2\leq X)}$, we move $\text{Re}(u)=3$ to $\text{Re}(u)=-1/5$ and $\text{Re}(v) = 3$ to $\text{Re}(v)=1/\log X$. Then
\begin{align*}
& T_{(N_1\leq M, M<N_2\leq X)} \ll \sumd_{\substack{N_1\leq M \\ M < N_2\leq X}} \log\bfrac{M}{N_1} \bfrac{1}{2\pi i}^4 \int_{(\frac{1}{\log X})} \int_{(0)} \int_{(0)} \gamma_g(v)   \frac{\frac{(X)^v-M^v}{N_2^v}}{v^2} \tilde{G}(w_1)\tilde{G}(w_2-v) N_1^{w_1}N_2^{w_2}\\ 
& \times  \sumstar_{(k_1,2)=1} \sum_{\substack{ k_2\geq 1 \\ (k_2,2)=1}} \sumtwo_{(n_1n_2,2a)=1} \frac{\lambda_f(n_1)\lambda_g(n_2)\chi_{b}(n_1n_2Q')}{n_1^{1/2+w_1}n_2^{1/2+w_2}} \frac{G_{k_1k_2^2}(n_1n_2Q')}{n_1n_2Q'}V\bfrac{n_1}{N_1}V\bfrac{n_2}{N_2}\\
& \times \int_{(1/2)} \tilde{F}(1+v-s) \Gamma(s) (\cos+\text{sgn}(k_1)\sin)\bfrac{\pi s}{2} \bfrac{\pi |k_1|k_2^2 X}{8a^2bn_1n_2Q'}^{-s} \;dw_1dw_2dvds \\
& \ll \sumd_{\substack{N_1\leq M \\ M < N_2\leq X}} \log\bfrac{M}{N_1} \bfrac{1}{2\pi i}^3 \int_{(\frac{1}{\log X})} \int_{(0)} \int_{(0)} \gamma_g(v)   \frac{\frac{(X)^v-M^v}{N_2^v}}{v^2} \tilde{G}(w_1)\tilde{G}(w_2-v) N_1^{w_1}N_2^{w_2}\\ 
& \times  \sumstar_{(k_1,2)=1} T(k_1,1/2+w_1,1/2+w_2,Q'; v) \;dw_1dw_2dv \\
& \ll \sumd_{\substack{N_1\leq M \\ M<N_2\leq X}} \log\bfrac{M}{N_1} \bfrac{X}{N_2}^{1/\log X} (\log\log X)^2  \int_{-\infty}^{\infty} \int_{-\infty}^{\infty} \int_{-\infty}^{\infty} \\
&  \frac{1}{(1+|\frac{1}{\log X}+it_4|)^{20}} \frac{1}{(1+|t_1|)^{20}} \frac{1}{(1+|-\frac{1}{\log X}+i(t_2-t_4)|)^{20}} \Bigg|\sumstar_{(k_1,2)=1} T\Bigg(k_1,1/2+it_1,1/2+it_2,Q',\frac{1}{\log X}+it_4\Bigg)\Bigg| \; dt_1dt_2dt_4,
\end{align*}

For $T_{(N_1\leq M, X<N_2)}$, we move $\text{Re}(u)=3$ to $\text{Re}(u)=-1/5$ and $\text{Re}(v) = 3$ to $\text{Re}(v)=6$. Then
\begin{align*}
& T_{(N_1\leq M, X< N_2)} \ll \sumd_{\substack{N_1\leq M \\  X<N_2}} \log\bfrac{M}{N_1} \bfrac{X}{N_2}^6  \int_{-\infty}^{\infty} \int_{-\infty}^{\infty} \int_{-\infty}^{\infty}  \frac{1}{(1+|6+it_4|)^{20}} \frac{1}{(1+|t_1|)^{20}} \frac{1}{(1+|-6+i(t_2-t_4)|)^{20}} \\
& \Bigg|\sumstar_{(k_1,2)=1} T(k_1,1/2+it_1,1/2+it_2,Q',6+it_4)\Bigg| \;dt_1dt_2dt_4.
\end{align*}

For $T_{(M<N_1, N_2\leq M)}$, move $\text{Re}(u)=3$ to $\text{Re}(u)=6$ and $\text{Re}(v)= 3$ to $\text{Re}(v) = -1/5$. Then
\begin{align*}
& T_{(M<N_1,N_2\leq M)} \ll \sumd_{\substack{M<N_1 \\ N_2\leq M}} \bfrac{M}{N_1}^6 \log\bfrac{X}{M} \int_{-\infty}^{\infty}  \int_{-\infty}^{\infty}  \int_{-\infty}^{\infty} \frac{1}{(1+|6+it_3|)^{20}} \frac{1}{(1+|-6+i(t_1-t_3)|)^{20}} \frac{1}{(1+|t_2|)^{20}} \\
& \Bigg|\sumstar_{(k_1,2)=1} T(k_1,1/2+it_1,1/2+it_2,Q')\Bigg| \;dt_1dt_2dt_3.
\end{align*}

For $T_{(M<N_1, M<N_2\leq X)}$, move $\text{Re}(u)=3$ to $\text{Re}(u)=6$ and $\text{Re}(v)= 3$ to $\text{Re}(v) = 1/\log X$. Then
\begin{align*}
& T_{(M<N_1, M<N_2\leq X)} \ll \sumd_{\substack{M<N_1 \\ M<N_2\leq X}} \bfrac{M}{N_1}^6 \bfrac{X}{N_2}^{1/\log X} (\log\log X)^2 \int_{-\infty}^{\infty}  \int_{-\infty}^{\infty}  \int_{-\infty}^{\infty} \int_{-\infty}^{\infty}  \frac{1}{(1+|6+it_3|)^{20}}  \frac{1}{(1+|\frac{1}{\log X}+it_4|)^{20}} \\
&  \frac{1}{(1+|-6+i(t_1-t_3)|)^{20}}  \frac{1}{(1+|-\frac{1}{\log X}+i(t_2-t_4)|)^{20}} \Bigg|\sumstar_{(k_1,2)=1} T\Bigg(k_1,1/2+it_1,1/2+it_2,Q', \frac{1}{\log X}+it_4\Bigg)\Bigg| \; dt_1dt_2dt_3dt_4.
\end{align*}

For $T_{(M<N_1, X<N_2)}$, move $\text{Re}(u),\text{Re}(v)=3$ to $\text{Re}(u),\text{Re}(v)=6$. Then
\begin{align*}
& T_{(M<N_1, X<N_2)} \ll \sumd_{\substack{M<N_1\\ X<N_2}} \bfrac{M}{N_1}^6 \bfrac{X}{N_2}^6 \int_{-\infty}^{\infty}  \int_{-\infty}^{\infty}  \int_{-\infty}^{\infty} \int_{-\infty}^{\infty}  \frac{1}{(1+|6+it_3|)^{20}}  \frac{1}{(1+|6+it_4|)^{20}} \\
&  \frac{1}{(1+|-6+i(t_1-t_3)|)^{20}}  \frac{1}{(1+|-6+i(t_2-t_4)|)^{20}} \Bigg|\sumstar_{(k_1,2)=1} T(k_1,1/2+it_1,1/2+it_2,Q',6+it_4)\Bigg| \; dt_1dt_2dt_3dt_4.
\end{align*}

By (\ref{eqn:standardbounds}), there is a smooth decay in $t_4$. In other words, we can use (\ref{eqn:fixed}) to conclude that $T_0(a,Q')$ is 
\begin{align*}
& \ll \frac{a^2b}{X}\Bigg(\;\sumd_{\substack{N_1\leq M \\ N_2\leq M}}\log\bfrac{M}{N_1}\log\bfrac{X}{M}\sqrt{N_1N_2} + \sumd_{\substack{N_1\leq M \\ M< N_2\leq X}}\log\bfrac{M}{N_1}\bfrac{X}{N_2}^{1/\log X} (\log\log X)^2 \sqrt{N_1N_2} +  \sumd_{\substack{N_1\leq M \\ X<N_2}}\log\bfrac{M}{N_1}\bfrac{X}{N_2}^{6}\sqrt{N_1N_2} \\
& +  \sumd_{\substack{ M<N_1 \\ N_2\leq M}}\bfrac{M}{N_1}^6\log\bfrac{X}{M}\sqrt{N_1N_2} +  \sumd_{\substack{ M<N_1 \\  M<N_2\leq X}}\bfrac{M}{N_1}^6 \bfrac{X}{N_2}^{1/\log X}(\log\log X)^2 \sqrt{N_1N_2} +  \sumd_{\substack{ M<N_1 \\ X<N_2}}\bfrac{M}{N_1}^6 \bfrac{X}{N_2}^6\sqrt{N_1N_2}\Bigg) \\
& \ll \frac{a^2b}{X} \sqrt{MX} (\log M)^2 (\log\log X)^3 \ll \frac{a^2b}{(\log X)^{48-\epsilon}}.
\end{align*}
Following (\ref{eqn:delta0}) and (\ref{eqn:delta1}), we also derive
\begin{align*}
T_{\ell}(a,Q') \ll \frac{a^2b}{2^{\ell} (\log X)^{48-\epsilon}}.
\end{align*} 
for any $\ell\geq 1$. Therefore, $T_{\text{off-diagonal}}$ is
\begin{align*}
\ll X \sum_{\substack{(a,2Q)=1 \\ a\leq Y}} \frac{1}{a^2} \sum_{b\mid Q} \frac{1}{b} \sum_{\ell\geq 0} \frac{a^2b}{2^{\ell} (\log X)^{48-\epsilon}} \ll \frac{YX}{(\log X)^{48-\epsilon}} \ll \frac{X}{(\log X)^{10}}
\end{align*} 
by choosing $Y = (\log X)^{20}$.

\section{Proof of Proposition \ref{prop:Third}}
Let $h_2(n,m,d)=W_f(n;|8d|,M)W_g(m;|8d|,M)F\bfrac{8d}{X}$ and then
\begin{align*}
\text{III}(f,g) & = S_{f,g}(1;h_2) - i^{\kappa_1}\eta_fS_{f,g}(q_1;h_2)-  i^{\kappa_2}\eta_gS_{f,g}(q_2;h_2)+ i^{\kappa_1+\kappa_2}\eta_f\eta_gS_{f,g}(Q;h_2),
\end{align*}
where $S_{f,g}(Q';h_2)$ is defined as in $(\ref{eqn:decomposition})$. We insert the smooth partition of unity and split each $N_i$ sum in three ranges
\begin{align*}
S_{f,g}(Q';h_2) & = \sumd_{N_1,N_2} \sumstar_{(d,2Q)=1}  \sumtwo_{n_1,n_2} \frac{\lambda_f(n_1)\lambda_g(n_2)}{\sqrt{n_1n_2}} \chi_{8d}(n_1n_2Q') h_2(n_1,n_2,d)G\bfrac{n_1}{N_1}G\bfrac{n_2}{N_2}\\
& = S_{(N_1\leq M,N_2\leq M)} + S_{(N_1\leq M, M<N_2\leq X)}  +S_{(N_1\leq M,X<N_2)}+  S_{(M<N_1\leq X,  N_2\leq M)}  + S_{(M< N_1\leq X, M<N_2\leq X)} \\
&+ S_{(M<N_1\leq X,  X< N_2)} + S_{(X< N_1, N_2\leq M)} + S_{(X< N_1, M<N_2\leq X)}+  S_{(X<N_1, X<N_2)}. 
\end{align*}
\subsection{Bound of $S_{(N_1\leq M, N_2)}$}
After applying M\"obius inversion, we split $S_{(N_1\leq M,N_2)}$ in two ranges
 \begin{align*}
T_{(a\leq Y)} = \sumd_{\substack{N_1\leq M \\ N_2}} \sum_{\substack{(a,2Q)=1 \\ a\leq Y}} \mu(a)\sum_{b\mid Q} \mu(b) \sum_{(d,2)=1}\sumtwo_{(n_1n_2,2a)=1} \frac{\lambda_f(n_1)\lambda_g(n_2)}{\sqrt{n_1n_2}} \chi_{8d}(n_1n_2Q') h_2(n_1,n_2,8a^2bd)G\bfrac{n_1}{N_2}G\bfrac{n_2}{N_2},
 \end{align*}
 and
 \begin{align*}
T_{(a> Y)} = \sumd_{\substack{N_1\leq M \\ N_2}} \sum_{\substack{(a,2Q)=1 \\ a> Y}} \mu(a)\sum_{b\mid Q} \mu(b) \sum_{(d,2)=1}\sumtwo_{(n_1n_2,2a)=1} \frac{\lambda_f(n_1)\lambda_g(n_2)}{\sqrt{n_1n_2}} \chi_{8d}(n_1n_2Q') h_2(n_1,n_2,8a^2bd)G\bfrac{n_1}{N_2}G\bfrac{n_2}{N_2},
 \end{align*}
 \subsection{Bound of $T_{(a>Y)}$} Apply Cauchy-Schwarz inequality and using the bound from \cite[Proposition 2]{KMSS}, we derive
 \begin{align*}
     T_{(a>Y)}\ll \frac{X(\log M)^{2+\epsilon}}{Y^{1-\epsilon}}.
 \end{align*}
 \subsection{Bound of $T_{(a\leq Y)}$} Apply Mellin inversion for $G$ and do a change of variables, then by (\ref{eqn:cutofffunction}), we write
\begin{align*}
T_{(a\leq Y)} & = \sumd_{\substack{N_1\leq M \\ N_2}} \sum_{\substack{(a,2Q)=1 \\ a\leq Y}} \mu(a) \sum_{b\mid Q} \mu(b)  \bfrac{1}{2\pi i}^4  \int_{(3)}\int_{(3)} \int_{(0)} \int_{(0)} \gamma_f(u)\gamma_g(v)\frac{1}{u^2} \frac{1}{v^2}\tilde{G}(w_1-u)\tilde{G}(w_2-v)  \\
&\times N_1^{w_1-u}N_2^{w_2-v}   \sumtwo_{(n_1n_2,2a)=1} \frac{\lambda_f(n_1)\lambda_g(n_2)\chi_{b}(n_1n_2Q')}{n_1^{1/2+w_1}n_2^{1/2+w_2}} \sum_{(d,2)=1} \chi_{8d}(n_1n_2Q') H_{u,v}\bfrac{8a^2bd}{X}\; dw_1dw_2dudv,
\end{align*}
where $H_{u,v}(s):=F(s)[(|s|X)^u-M^u][(|s|X)^v-M^v]$. Apply Poisson summation and then by (\ref{eqn:mainterm}), we extract the main term $\mathcal{M}_{Q'}$ and the error term $\mathcal{E}_{Q'}$
\begin{align*}
    \mathcal{M}_{Q'} &  = \frac{X}{2\pi^2} \sumd_{\substack{N_1\leq M \\ N_2}} \bfrac{1}{2\pi i}^4 \int_{(3)}\int_{(3)}\int_{(0)}\int_{(0)} \gamma_f(u)\gamma_g(v)\frac{1}{u^2}\frac{1}{v^2}\tilde{G}(w_1-u)\tilde{G}(w_2-v)N_1^{w_1-u}N_2^{w_2-v} \\
    & \sumtwo_{\substack{(n_1n_2,2a)=1 \\ n_1n_2Q'=\square}} \frac{\lambda_f(n_1)\lambda_g(n_2)}{n_1^{1/2+w_1}n_2^{1/2+w_2}}\prod_{p\mid n_1n_2Q} \frac{p}{p+1}\check{H_{u,v}}(0)\; dw_1dw_2dudv,
\end{align*}
and
\begin{align*}
    \mathcal{E}_{Q'} &  \ll \frac{X}{16Y} \sumd_{\substack{N_1\leq M \\ N_2}} \bfrac{1}{2\pi i}^4 \int_{(3)}\int_{(3)}\int_{(0)}\int_{(0)} \gamma_f(u)\gamma_g(v)\frac{1}{u^2}\frac{1}{v^2}\tilde{G}(w_1-u)\tilde{G}(w_2-v)N_1^{w_1-u}N_2^{w_2-v} \\
    & \sumtwo_{\substack{(n_1n_2,2a)=1 \\ n_1n_2Q'=\square}} \frac{\lambda_f(n_1)\lambda_g(n_2)}{n_1^{1/2+w_1}n_2^{1/2+w_2}} \prod_{p\mid n_1n_2Q} \frac{p}{p+1} \check{H_{u,v}}(0)\; dw_1dw_2dudv \ll \frac{X(\log X)^{11}}{Y}.
\end{align*}
By (\ref{eqn:fouriertype1}) and (\ref{eqn:mainterm}), we further have
\begin{align*}
    M_{Q'} & = \frac{X}{2\pi^2} \sumd_{\substack{N_1\leq M \\ N_2}} \int_{-\infty}^{\infty} F(x) \bfrac{1}{2\pi i}^4 \int_{(3)}\int_{(3)}\int_{(0)}\int_{(0)} \gamma_f(u)\gamma_g(v)\frac{(|x|X)^u-M^u}{u^2}\frac{(|x|X)^v-M^v}{v^2}\tilde{G}(w_1-u)\tilde{G}(w_2-v)N_1^{w_1-u}N_2^{w_2-v} \\
    & L(1+w_1+w_2,f\otimes g)L(1+2w_1,\text{sym}^2f) L(1+2w_2,\text{sym}^2g)Z_{Q'}^*(w_1,w_2)\; dw_1dw_2dudvdx.
\end{align*}
Move $\text{Re}(u),\text{Re}(v)=3$ to $\text{Re}(u),\text{Re}(v)=-1/10$, where we encounter simple poles at $u,v=0$. Then move the lines of integration from $\text{Re}(w_1),\text{Re}(w_2)=0$ to $\text{Re}(w_1),\text{Re}(w_2) = -1/5$. Then Cauchy's Theorem implies
\begin{align*}
    M_{Q'}\ll X(\log \log X)^2 + O(X^{9/10+\epsilon}),
\end{align*}
where we have used the fact that $\sumd_{N} N^{-\epsilon} \ll 1$ for any $\epsilon>0$.

\subsubsection{(Off-diagonal terms $k\neq 0$)}
Now consider
\begin{align*}
    T_{\text{off-diagonal}} & = \frac{X}{16} \sum_{\substack{(a,2Q)=1 \\ a\leq Y}}\frac{\mu(a)} {a^2}\sum_{b\mid Q} \frac{\mu(b)}{b}  \sumd_{\substack{N_1\leq M \\ N_2}} \bfrac{1}{2\pi i}^4  \int_{(3)}\int_{(3)} \int_{(0)} \int_{(0)} \gamma_f(u)\gamma_g(v)\frac{1}{u^2} \frac{1}{v^2}\tilde{G}(w_1-u)\tilde{G}(w_2-v) N_1^{w_1-u}N_2^{w_2-v}   \\
& \sum_{k\neq 0} (-1)^k  \sumtwo_{(n_1n_2,2a)=1} \frac{\lambda_f(n_1)\lambda_g(n_2)\chi_{b}(n_1n_2Q')}{n_1^{1/2+w_1}n_2^{1/2+w_2}} \frac{G_k(n_1n_2Q')}{n_1n_2Q'} \check{H_{u,v}}\bfrac{kX}{16a^2bn_1n_2Q'}\; dw_1dw_2dudv.
\end{align*}
We next rewrite the sums
\begin{align*}
    \sum_{k\neq 0} \sumtwo_{(n_1n_2,2a)=1} = \sum_{\ell\geq 1} T_{\ell}(a,Q') -T_0(a,Q')
\end{align*}
with
\begin{align*}
    T_{\ell}(a,Q') & = \sumd_{\substack{N_1\leq M \\ N_2}} \bfrac{1}{2\pi i}^4  \int_{(3)}\int_{(3)} \int_{(0)} \int_{(0)} \gamma_f(u)\gamma_g(v)\frac{1}{u^2} \frac{1}{v^2}\tilde{G}(w_1-u)\tilde{G}(w_2-v) N_1^{w_1-u}N_2^{w_2-v}   \\
& \sum_{(k,2)=1}  \sumtwo_{(n_1n_2,2a)=1} \frac{\lambda_f(n_1)\lambda_g(n_2)\chi_{b}(n_1n_2Q')}{n_1^{1/2+w_1}n_2^{1/2+w_2}} \frac{G_{2^{\delta(\ell)}k}(n_1n_2Q')}{n_1n_2Q'} \check{H_{u,v}}\bfrac{2^{\ell}kX}{16a^2bn_1n_2Q'}\; dw_1dw_2dudv.
\end{align*}
and $\delta(\ell)=1$ if $2\nmid \ell$ and $0$ otherwise. It suffices to bound $T_{0}(a,Q')$. Redo the Mellin inversion for $G$ and insert the function $V$ over $n_i$ sums. Next write $k=k_1k_2^2$ with $k_1$ squarefree and $k_2$ positive odd integer and apply Mellin inversion for $G$ once more. By (\ref{eqn:fouriertype}), we have
\begin{align*}
T_{0}(a,Q') &  = \sumd_{\substack{N_1\leq M \\ N_2}} \bfrac{1}{2\pi i}^5 \int_{(1/2)} \int_{0}^{\infty} F(x)x^{-s} \int_{(3)} \int_{(3)} \int_{(0)} \int_{(0)} \gamma_f(u)\gamma_g(v) \frac{\frac{(|x|X)^u-M^u}{N_1^u}}{u^2}  \frac{\frac{(|x|X)^v-M^v}{N_2^v}}{v^2} \tilde{G}(w_1-u)\tilde{G}(w_2-v) N_1^{w_1}N_2^{w_2}\\ 
&\sumstar_{(k_1,2)=1} \sum_{\substack{ k_2\geq 1 \\ (k_2,2)=1}} \sumtwo_{(n_1n_2,2a)=1} \frac{\lambda_f(n_1)\lambda_g(n_2)\chi_{b}(n_1n_2Q')}{n_1^{1/2+w_1}n_2^{1/2+w_2}} \frac{G_{k_1k_2^2}(n_1n_2Q')}{n_1n_2Q'}V\bfrac{n_1}{N_1}V\bfrac{n_2}{N_2}  \Gamma(s)  (\cos+\text{sgn}(k_1)\sin)\bfrac{\pi s}{2} \\
& \bfrac{\pi |k_1|k_2^2 X}{8a^2bn_1n_2Q'}^{-s} \; dw_1dw_2dudvdxds \\
& = T_{(N_1\leq M,N_2\leq M)} + T_{(N_1\leq M, M<N_2\leq X)} +T_{(N_1\leq M, X<N_2)},
\end{align*}
where we split $N_2$ sum in three ranges. Move $\text{Re}(u),\text{Re}(v)=3$ to $\text{Re}(u),\text{Re}(v)=-1/5$, then $T_{(N_1\leq M,N_2\leq M)}$ is 
\begin{align*}
\ll \sumd_{\substack{N_1\leq M \\N_2\leq M}} \log^2\bfrac{X}{M} \int_{-\infty}^{\infty} \int_{-\infty}^{\infty} \frac{1}{(1+|t_1|)^{20}} \frac{1}{(1+|t_2|)^{20}} \Bigg|\sumstar_{(k_1,2)=1} T(k_1,1/2+it_1,1/2+it_2,Q')\Bigg| \; dt_1dt_2,
\end{align*}
where $T(k_1,1/2+it_1,1/2+it_2,Q')$ is defined as in (\ref{eqn:T1}). 

 For $T_{(N_1\leq M, M<N_2\leq X)}$, we move $\text{Re}(u)=3$ to $\text{Re}(u)=-1/5$ and $\text{Re}(v) = 3$ to $\text{Re}(v)=1/\log X$. Then
\begin{align*}
& T_{(N_1\leq M, M<N_2\leq X)} \ll \sumd_{\substack{N_1\leq M \\ M<N_2\leq X}} \log\bfrac{X}{M} \bfrac{X}{N_2}^{1/\log X} (\log\log X)^2  \int_{-\infty}^{\infty} \int_{-\infty}^{\infty} \int_{-\infty}^{\infty} \\
&  \frac{1}{(1+|\frac{1}{\log X}+it_4|)^{20}} \frac{1}{(1+|t_1|)^{20}} \frac{1}{(1+|-\frac{1}{\log X}+i(t_2-t_4)|)^{20}} \Bigg|\sumstar_{(k_1,2)=1} T\Bigg(k_1,1/2+it_1,1/2+it_2,Q',\frac{1}{\log X}+it_4\Bigg)\Bigg| \;dt_1dt_2dt_4.
\end{align*}

For $T_{(N_1\leq M, X<N_2)}$, we move $\text{Re}(u)=3$ to $\text{Re}(u)=-1/5$ and $\text{Re}(v) = 3$ to $\text{Re}(v)=6$. Then we have
\begin{align*}
& T_{(N_1\leq M, X< N_2)} \ll \sumd_{\substack{N_1\leq M \\  X<N_2}} \log\bfrac{X}{M} \bfrac{X}{N_2}^6  \int_{-\infty}^{\infty} \int_{-\infty}^{\infty} \int_{-\infty}^{\infty}  \frac{1}{(1+|6+it_4|)^{20}} \frac{1}{(1+|t_1|)^{20}} \frac{1}{(1+|-6+i(t_2-t_4)|)^{20}} \\
& \Bigg|\sumstar_{(k_1,2)=1} T(k_1,1/2+it_1,1/2+it_2,Q',6+it_4)\Bigg| \;dt_1dt_2dt_4.
\end{align*}
Combine all of the contributions above and by (\ref{eqn:fixed}), we derive that $T_0(a,Q')$ is 
\begin{align*}
& \ll \frac{a^2b}{X}\Bigg(\;\sumd_{\substack{N_1\leq M \\ N_2\leq M}}\log^2\bfrac{X}{M}\sqrt{N_1N_2} + \sumd_{\substack{N_1\leq M \\ M< N_2\leq X}}\log\bfrac{X}{M}\bfrac{X}{N_2}^{1/\log X} (\log\log X)^2 \sqrt{N_1N_2} +  \sumd_{\substack{N_1\leq M \\ X<N_2}}\log\bfrac{X}{M}\bfrac{X}{N_2}^{6}\sqrt{N_1N_2}\Bigg)  \\
& \ll \frac{a^2b}{(\log X)^{49-\epsilon}}.
\end{align*}
Following (\ref{eqn:delta0}) and (\ref{eqn:delta1}), we also have
\begin{align*}
T_{\ell}(a,Q') \ll \frac{a^2b}{2^{\ell}(\log X)^{49-\epsilon}}
\end{align*}
for any $\ell\geq 1$. Therefore, we can conclude that
\begin{align*}
T_{\text{off-diagonal}} \ll X \sum_{\substack{(a,2Q)=1 \\ a\leq Y}} \frac{1}{a^2} \sum_{b\mid Q} \frac{1}{b} \sum_{\ell\geq 0} \frac{a^2b}{2^{\ell} (\log X)^{49-\epsilon}} \ll \frac{YX}{(\log X)^{49-\epsilon}} \ll \frac{X}{(\log X)^{10}}
\end{align*} 
again by choosing $Y= (\log X)^{20}$. By symmetry, we also have
\begin{align*}
S_{(N_2\leq M,N_1)} \ll X(\log\log X)^2.
\end{align*}

\subsection{Bound of $S_{(M\leq N_1<X, M\leq N_2<X)}$} Apply Cauchy-Schwarz inequality and we have
\begin{align*}
S_{(M\leq N_1<X, M\leq N_2<X)} & \leq \sumd_{M<N_1\leq X} \Bigg(\sumstar_{(d,2Q)=1} F\bfrac{8d}{X} \Bigg| \sum_{n_1} \frac{\lambda_f(n_1)\chi_{8d}(n_1)}{\sqrt{n_1}}W_f(n_1;|8d|,M)G\bfrac{n_1}{N_1}\Bigg|^2\Bigg)^{1/2} \\
& \times \sumd_{M<N_2\leq X} \Bigg(\sumstar_{(d,2Q)=1} F\bfrac{8d}{X} \Bigg| \sum_{n_2} \frac{\lambda_g(n_2)\chi_{8d}(n_2)}{\sqrt{n_2}}W_g(n_2;|8d|,M)G\bfrac{n_2}{N_2}\Bigg|^2\Bigg)^{1/2}.
\end{align*}
\begin{lemma}\label{lem:2B} With the notation above, we have
\begin{align*}
\sumd_{M<N\leq X} \Bigg(\sumstar_{(d,2Q)=1} F\bfrac{8d}{X} \Bigg| \sum_{n} \frac{\lambda_f(n)\chi_{8d}(n)}{\sqrt{n}}W_f(n;|8d|,M)G\bfrac{n}{N}\Bigg|^2\Bigg)^{1/2} \ll X^{1/2}(\log\log X)^3.
\end{align*}
\end{lemma}
\begin{proof} 
Let 
\begin{align*}
S_1: = \sumstar_{(d,2Q)=1} F\bfrac{8d}{X} \Bigg| \sum_{n} \frac{\lambda_f(n)\chi_{8d}(n)}{\sqrt{n_1}}W_f(n;|8d|,M)G\bfrac{n}{N}\Bigg|^2.
\end{align*}
Following (\ref{eqn:cutofffunction}), we insert function $V$ and then apply Mellin inversion for $G$ and do a change of variable to get
\begin{align*}
S_1 =  \sumstar_{(d,2Q)=1}F\bfrac{8d}{X} \left| \bfrac{1}{2\pi i}^2 \int_{(3)} \int_{(0)}\gamma_f(u) \frac{\bfrac{|8d|}{N}^u-\bfrac{M}{N}^u}{u^2}  \tilde{G}(s-u) N^s \sum_{n} \frac{\lambda_f(n)\chi_{8d}(n)}{n^{1/2+s}}V\bfrac{n}{N} \; ds du \right|^2.
\end{align*}
Move $\text{Re}(u)=3$ to $\text{Re}(u) = 1/\log X$ and apply Cauchy-Schwarz inequality, we have
\begin{align*}
S_1 & \leq  \sumstar_{(d,2Q)=1} F\bfrac{8d}{X} \bfrac{X}{N}^{2/\log X}\int_{( 1/\log X)}\int_{(0)}|\gamma_f(u)|  \left|\frac{\bfrac{|8d|}{M}^u-1}{u^{3/2}}\right|^2 |\tilde{G}(s-u)|  \;dsdu  \\
& \times   \int_{( 1/\log X)}\int_{(0)} \frac{1}{|u|} |\gamma_f(u)| |\tilde{G}(s-u)| \left|\sum_{n} \frac{\lambda_f(n)\chi_{8d}(n)}{n^{{1/2}+s}}V\bfrac{n}{N}\right|^2\; dsdu.
\end{align*}
Note that the first double integral is
\begin{align} \label{eqn:FirstDoubleIntegral}
&\ll \int_{-\infty}^{\infty} \int_{-\infty}^{\infty}  \frac{1}{|\frac{1}{\log X}+it_1|} \frac{1}{(1+|\frac{1}{\log  X}+it_1|)^{20}}  \frac{1}{(1+|-\frac{1}{\log  X}+i(t_2-t_1)|)^{20}}  \left|\frac{\bfrac{|8d|}{M}^{\frac{1}{\log X}+it_1}-1}{\frac{1}{\log X}+it_1}\right|^2  \;dt_2dt_1,
\end{align}
where the last factor is uniformly bounded by $(\log\log X)^2$. To justify this, we split $t_1$ in two ranges: $|t_1| \ll \frac{1}{\log \log X}$ and $\frac{1}{\log \log X} \ll |t_1|$. When $|t_1|\ll\frac{1}{\log\log X}$, we have
\begin{align*}
\left| \frac{1}{\log X}+it_1\right| & \ll \sqrt{\bfrac{1}{\log X}^2+ \bfrac{1}{\log\log X}^2}  \ll \sqrt{\bfrac{1}{\log\log X}^2} = \frac{1}{\log\log X},
\end{align*}
so then
\begin{align*}
\left|\frac{\bfrac{|8d|}{M}^{\frac{1}{\log X}+it_1}-1}{\frac{1}{\log X}+it_1}\right| = \left|\frac{\exp\left({\left(\frac{1}{\log X}+it_1\right)\log\bfrac{|8d|}{M}}\right)-1}{\frac{1}{\log X}+it_1}\right| \leq (e-1) \left|\frac{\left(\frac{1}{\log X}+it_1\right)\log\bfrac{|8d|}{M}}{\frac{1}{\log X}+it_1}\right| \ll\log\bfrac{|8d|}{M} \ll \log\log X.
\end{align*}
When $\frac{1}{\log\log X} \ll |t_1|$, the triangle inequality gives us
\begin{align*}
\left|\frac{\bfrac{|8d|}{M}^{\frac{1}{\log X}+it_1}-1}{\frac{1}{\log X}+it_1}\right| & \leq \frac{\bfrac{|8d|}{M}^{\frac{1}{\log X}}+1}{\sqrt{\bfrac{1}{\log X}^2+t_1^2}} \ll \frac{2}{\sqrt{\bfrac{1}{\log X}^2+\bfrac{1}{\log\log X}^2}} \leq \frac{2}{\frac{1}{\log\log X}} \ll \log\log X.
\end{align*}

\begin{lemma}\label{lem:IntegralBound}
Let $\delta>0$ be a fixed real number and $k>0$, we have
\[
\int_{-\infty}^{\infty} \frac{1}{\delta+|t|} \frac{1}{(1+|t|)^k} \; dt \ll \log\bfrac{1}{\delta}.
\]
\end{lemma}
\begin{proof}
This is an application of integration by parts. In particular, we have
\begin{align*}
\int_{-\infty}^{\infty} \frac{1}{\delta+|t|} \frac{1}{(1+|t|)^k}  \;dt & = 2\cdot \int_{0}^{\infty} \frac{1}{\delta+t} \frac{1}{(1+t)^k} \; dt \\
& = 2 \cdot \left(\log(\delta+t)\cdot\frac{1}{(1+t)^k}\Bigg|_{0}^{\infty} + k \cdot \int_{0}^{\infty} \log(\delta+t) \cdot \frac{1}{(1+t)^{k+1}} \; dt\right) \\ 
& = 2\cdot  \left(-\log(\delta) + k \cdot \int_{0}^{\infty} \log(\delta+t) \cdot \frac{1}{(1+t)^{k+1}} \; dt\right) \\
& \ll -\log(\delta),
\end{align*}
where the improper integral converges absolutely for any $k>0$ since $\log(\delta+t) \ll t^{\epsilon}$ for any $\epsilon>0$.
\end{proof}
Lemma \ref{lem:IntegralBound} implies that the first double integral is $\ll (\log\log X)^3$. Apply the bounds from Lemma \ref{lem:NonDyadicCut} and Lemma \ref{lem:IntegralBound}, the second double integral is $\ll X\log\log X$ and
\begin{align*} 
&S_1 \ll \bfrac{X}{N}^{2/\log X} X(\log \log X)^4.
\end{align*}
Since $\sumd_{M<N\leq X} \bfrac{X}{N}^{1/\log X} \ll \log\log X$ and we finished proving the lemma.
\end{proof}

Using Lemma \ref{lem:2B}, we conclude
\begin{align*}
S_{(M<N_1\leq X, M<N_2\leq X)} \ll X(\log\log X)^6.
\end{align*}

\subsection{Bound of $S_{(M<N_1\leq X, X<N_2)}$} Apply Cauchy-Schwarz inequality, then
\begin{align*}
 S_{(M\leq N_1<X, X< N_2)} & \leq \sumd_{M<N_1\leq X} \Bigg(\sumstar_{(d,2Q)=1} F\bfrac{8d}{X} \Bigg| \sum_{n_1} \frac{\lambda_f(n_1)\chi_{8d}(n_1)}{\sqrt{n_1}}W_f(n_1;|8d|,M)G\bfrac{n_1}{N_1}\Bigg|^2\Bigg)^{1/2} \\
& \times \sumd_{X<N_2} \Bigg(\sumstar_{(d,2Q)=1} F\bfrac{8d}{X} \Bigg| \sum_{n_2} \frac{\lambda_g(n_2)\chi_{8d}(n_2)}{\sqrt{n_2}}W_g(n_2;|8d|,M)G\bfrac{n_2}{N_2}\Bigg|^2\Bigg)^{1/2}.
\end{align*}

For the second factor, we have
\begin{align*}
S_2 & := \sumstar_{(d,2Q)=1} F\bfrac{8d}{X} \Bigg| \sum_{n} \frac{\lambda_g(n_2)\chi_{8d}(n)}{\sqrt{n}}W_g(n;|8d|,M)G\bfrac{n}{N}\Bigg|^2 \\
&  = \sumstar_{(d,2Q)=1} F\bfrac{8d}{X} \left| \int_{(3)}\int_{(0)} \gamma_g(u)  \frac{\bfrac{|8d|}{N}^u-\bfrac{M}{N}^u}{u^2}  \tilde{G}(s-u) N^s \sum_{n} \frac{\lambda_g(n)\chi_{8d}(n)}{n^{{1/2}+s}}V\bfrac{n}{N}\; dsdu\right|^2 \\
& \ll \bfrac{X}{N}^6 \int_{-\infty}^{\infty}\int_{-\infty}^{\infty} \frac{1}{|3+it_1|} \frac{1}{(1+|3+it_1|)^{20}} \frac{1}{(1+|-3+i(t_2-t_1)|)^{20}}  \; dt_2dt_1\\
& \times  \int_{-\infty}^{\infty}\int_{-\infty}^{\infty} \frac{1}{|3+it_1|} \frac{1}{(1+|3+it_1|)^{20}} \frac{1}{(1+|-3+i(t_2-t_1)|)^{20}}  \sum_{(d,2)=1}F\bfrac{8d}{X} \left|\sum_{n} \frac{\lambda_g(n)\chi_{8d}(n)}{n^{{1/2}+it_2}}V\bfrac{n}{N}\right|^2\; dt_2dt_1 \\
& \ll \bfrac{X}{N}^6 X.
\end{align*}

Since $\sumd_{X<N} \bfrac{X}{N}^3 \ll 1$ and the second factor is $\ll X^{1/2}$. Therefore, we conclude that
\begin{align*}
S_{(M<N_1\leq X, X<N_2)}, S_{(M<N_2\leq X, X<N_1)} & \ll X(\log\log X)^3, \\
S_{(X<N_1,X<N_2)} & \ll X.
\end{align*}

\end{document}